\documentclass[a4paper,11pt,english]{article}
\usepackage{amssymb}
\usepackage{mathrsfs}
\usepackage{amsmath}
\usepackage{amsthm}
\usepackage{amstext}
\usepackage{amsopn}
\usepackage{graphicx}
\usepackage{latexsym}
\usepackage{amsfonts,euscript}

\usepackage{color}

\newtheorem{theorem}{Theorem}[section]

\newtheorem{lemma}{Lemma}[section]

\newtheorem{remark}{Remark}[section]
\numberwithin{equation}{section} \textwidth 150mm \textheight 222mm
\makeatletter

\newcommand{\Rmnum}[1]{\expandafter\@slowromancap\romannumeral #1@}
\makeatother

\newcommand{\f}{\frac}
\newcommand{\R}{\mathbb {R}}

\newcommand{\mc}{\mathcal}

\begin{document}

\title{Propagation dynamics of a reaction-diffusion equation in a time-periodic shifting environment \thanks{The research leading to these results has received financial support from NSF of China and NSERC of Canada. }}
\author{Jian Fang\thanks{Institute for Advanced Study in Mathematics and Department of Mathematics, Harbin Institute of Technology, Harbin, 150001, China. Email: jfang@hit.edu.cn}, \, Rui Peng\thanks{Department of Mathematics and Statistics, Jiangsu Normal University, Xuzhou, 221116, China. Email: pengrui\_seu@163.com} \,  and \, Xiao-Qiang Zhao\thanks{Department of Mathematics and Statistics, Memorial University of Newfoundland, St. John's, NL  A1C 5S7, Canada. Email: zhao@mun.ca}}

\maketitle

\begin{abstract}
This paper concerns the nonautonomous reaction-diffusion equation
\[
u_t=u_{xx}+ug(t,x-ct,u), \quad t>0,x\in\R, 
\]
where $c\in\R$ is the shifting speed,  and the time periodic nonlinearity $ug(t,\xi,u)$  is  asymptotically of KPP type as $\xi \to-\infty$ and is negative as $\xi\to+\infty$. Under a subhomogeneity condition, we show that there is $c^*>0$ such that a unique forced time periodic wave exists if and only $|c|< c^*$ and it attracts other solutions in a certain sense  according to
the tail behavior of initial values. In the case where  $|c|\ge c^*$, the propagation dynamics resembles that of the limiting system as $\xi\to\pm \infty$, depending on the shifting direction.
\end{abstract}



\section{Introduction}
In this paper, we are interested in  the following nonautonomous  reaction-diffusion  equation in a time-periodic shifting environment:
\begin{equation}\label{main-eq}
\begin{cases}
u_t=u_{xx}+ug(t,x-ct,u),\quad  t>0,x\in\R,\\
u(0,\cdot)=\phi,
\end{cases}
\end{equation}
where $c\in\R$ is the shifting speed and $\phi$ is a bounded and  continuous function. A prototypical function is  $g(t,x,u)=a(t,x)-u$,  which  makes 
\eqref{main-eq} become the KPP-Fisher equation.
This type of equations models the population growth in a shifting media. They may arise from the biological question whether the species can survive from the climate change \cite{BDNZ-BMB09, LBSF-SIAP14, PM-BMB04}. Subject to seasonal succession, climate change provides such a shifting and time periodic environment for the species. More precisely, if $u$ represents the species density, then $x-ct$, a variable of the net per capita growth rate $g(t,x-ct,u)$, can be understood as the functional response to the environmental shifting. Such a nonlinearity may also arise from the epidemiological question whether the pathogen spread can keep pace with its host \cite{FLW-SIAP16}. We will derive a model equation from the pathogen spread viewpoint, as an example of \eqref{main-eq}, in the application section.

Throughout the whole  paper, we make the following assumptions on the function $g$:
\begin{enumerate}
\item[(G1)] $g\in C^1(\R\times \R\times \R_+,\R)$  and $g(t,x,u)$ is $T$-periodic in $t$
 for some $T>0$;
\item[(G2)]
$g(t,x,u)$ is non-increasing in $x\in \R$ and $u\in \R_+ $,  $g(t,-\infty, u)$ 
exists and is strictly decreasing in $u\in \R_+$;
\item[(G3)]$\int_0^T g(t,-\infty,0)dt>0$ and there exists $M>0$ such that 
$g(t,-\infty, M)\leq 0$;
\item[(G4)]$g(t,+\infty, u)$ exists and $\int_0^T g(t,+\infty,0)dt<0$.
\end{enumerate}

Assumptions (G3) and  (G4) imply that the scenario that the environment is favorable at $-\infty$ and unfavorable at $+\infty$. The sign of the shifting speed $c$ determines whether the favorable environment can invade the unfavorable one or the reverse.  It easily follows from (G3) that  the function $ug(t,-\infty,u)$ is of KPP type. As such, at $-\infty$ one has the limiting equation of KPP type
\begin{equation}\label{limiting}
u_t=u_{xx}+ug(t,-\infty,u).
\end{equation}
In view of \cite[Theorem 5.2.1]{ZhaoBook} with $F(t,u)\equiv F_0(t,u)$  and
\cite[Lemma 2.2.1]{ZhaoBook},
the corresponding ordinary differential equation
\begin{equation}\label{ODE}
u'(t)=ug(t,-\infty,u)
\end{equation}
has a globally stable positive periodic solution $\alpha(t)$ in $\R_+\setminus \{0\}$. By \cite[Theorems 4.1 and 4.2]{LYZ2006}
(letting $\tau=0$ and $f(t,u,v)\equiv f(t,u)$),  it follows that 
 the periodic reaction-diffusion equation \eqref{limiting} admits a  spreading speed $c^*$, and $c^*$ is also the minimal speed of time periodic positive traveling waves connecting $\alpha(t)$ to $0$. Moreover, 
  \cite[Lemma 4.1]{LYZ2006} implies that 
\begin{equation}\label{def-cstar}
c^*=2\sqrt{\f{1}{T}\int_0^T g(t,-\infty,0)dt}.
\end{equation}
While at  $+\infty$, one has a different limiting equation 
\begin{equation}\label{limiting2}
u_t=u_{xx}+ug(t,+\infty,u).
\end{equation}
Since $g(t,+\infty,u)\leq g(t,+\infty,0), \, \forall (t,u)\in \R^2_+$, it follows 
from the comparison argument and assumption (G4) that  every nonnegative solution of \eqref{limiting2} converges to zero uniformly for $x\in \R$ as $t\to +\infty$.

The purpose of this paper is to explore how these two limiting equations \eqref{limiting} and \eqref{limiting2} as well as the shifting speed $c$ influence the propagation dynamics of \eqref{main-eq}. 
 Before presenting  our main results, we briefly  review  
some related works,  which highly motived our current research.

Berestycki, Diekmann, Nagelkerke and Zegeleing \cite{BDNZ-BMB09} introduced the following KPP type equation
\begin{equation}\label{auto-eq}
u_t=u_{xx}+f(x-ct, u)
\end{equation} 
to study the impact of climate shift on the dynamics of a biological species. In \cite{BDNZ-BMB09}, assuming that favorable environment is surrounded by unfavorable ones, i.e., 
\begin{equation}\label{cc1}
\text{$f(x,\cdot)<0$ when $|x|$ is large},
\end{equation}
they showed that the global dynamics is determined by the sign of the generalized eigenvalue $\lambda_1$, which is defined as 
\begin{equation}
\lambda_1:=\sup\{\lambda: \exists \phi\in  C^2(\R), \phi>0, s.t.  \phi''+c\phi'+\partial_u f(x,0)\phi+\lambda\phi\le 0\}.
\end{equation}
More precisely, if $\lambda_1\ge 0$ then the solution goes to zero. While if $\lambda_1<0$ then the solution converges uniformly in $x\in\R$ to the unique solution of $U''+cU'+f(x,U)=0$. Such a result was also established for high dimensional and mixed type environments by Berestycki and Rossi in \cite{BR-DCDS08, BR-DCDS09}. Without condition \eqref{cc1}, the sign of $\lambda_1$ cannot determine the global dynamics of \eqref{auto-eq}. In particular,  Under the condition that $f(x,\cdot)$ is positive when $x\to-\infty$, Berestycki and Fang \cite{BF18} showed that $\lambda_1<0$ is the necessary and sufficient condition for the existence of the minimal positive solution of $U''+cU'+f(x,U)=0$.  This property,  together with the classification of solutions of $U''+cU'+f(x,U)=0$, was then used to derive  the global dynamics of \eqref{auto-eq}. The study in \cite{BF18} is closely related to some questions raised in \cite{BR-JEMS06, H-AIHP97-2}.  In \cite{BR-DCDS08, BR-DCDS09}, the authors also investigated  \eqref{main-eq} by assuming the sign of eigenvalue $\lambda_1$ of the related time periodic operator and the uniqueness of forced time periodic waves (see \cite[Theorem 3.6]{BR-DCDS08}). From these literatures, we have seen that $\lambda_1<0$ is sufficient for the existence of at least one solution of $U''+cU'+f(x,U)=0$, but the sign of $\lambda_1$ cannot determine the non-existence or uniqueness of such solutions. 

It is remarkable that the generalized eigenvalues in unbounded domains are useful tools in the study of propagation dynamics. Some different generalized eigenvalues were introduced and deeply investigated in \cite{BHR-AMRA07, BR-JEMS06, BR-CPAM15}.

Recently there is an increasing interest in the study of the influence of shifting environment on biological invasions. Zhou and Kot  \cite{ZK-Springer13}  introduced a class of integro-difference equations to model the effects of climate-driven range shifts (see also \cite{BL-preprint, ZhouKot-14, LMS, LBBF16}).  Du, Wei and Zhou \cite{DWZ-preprint} proposed a free boundary problem in such a shifting environment, see also \cite{LeiDu-DCDS17, WeiZhangZhou-CVPDE16}. Hu and Li \cite{HL-JDE15} formulated such a problem in a discrete media. Vo \cite{Vo-JDE15} investigated the persistence of species facing a forced time periodic and locally favorable environment in a cylindrical or partially periodic domain and established various  results on the existence and uniqueness of the forced waves. Wang, Li and Zhao \cite{LWZ18} studied the propagation
dynamics of a nonlocal dispersal equation in a shifting environment, see also \cite{ABR-preprint, WZ18}. Bouhours and Giletti  \cite{BG-2016preprint}  studied a generalized monostable reaction-diffusion equation in shifting environment, including the scenario with Allee effect. It is worthy to point out that a shifting environment can also arise in other ways. For example, Holzer and Scheel \cite{HS-SIAM14} considered a partially decoupled reaction diffusion system of two equations, where a wave solution for the first equation provides a shifting environment for the second one, see also \cite{CTW17, FLW-SIAP16}.

In our study of the global dynamics of \eqref{main-eq}, the positive time periodic wave solutions of \eqref{main-eq} having the form $u(t,x)=U(t,x-ct)$ will play an important role, where $c$ is the shifting speed given in \eqref{main-eq}. Clearly, $U(t,x)$ satisfies
\begin{equation}\label{TW}
U_t=U_{xx}+cU_x+Ug(t,x,U), \quad t\in\R,x\in\R
\end{equation}
with the periodic constraints
\begin{equation}
U(t+T, x)=U(t,x).
\end{equation}
By  some {\it a priori} estimates, we will have $0<U<\alpha(t)$ and $U(t,+\infty)=0$, where $\alpha(t)$ is the unique positive time periodic solution of \eqref{ODE}, see Lemma \ref{est}.

We call $U$ a {\it forced KPP wave} of \eqref{main-eq} if 
\begin{equation}\label{BC1}
U(t,-\infty)=\alpha(t),\quad U(t,+\infty)=0
\end{equation}
and a {\it forced pulse wave} of \eqref{main-eq}  if 
\begin{equation}\label{BC2}
U(t,-\infty)=0, \quad U(t,+\infty)=0,
\end{equation}
where all the limits hold uniformly for all $t\in [0,T]$.

Let  $c^*$ be defined as in \eqref{def-cstar}. Our first result  is about the existence, uniqueness and nonexistence of the forced KPP waves.

\begin{theorem}\label{ThTW}
The forced wave $U(t,x-ct)$ exists if and only if $c<c^*$, and such a wave  is  unique when exists.
\end{theorem} 

At $-\infty$, \eqref{main-eq} behaves as the time periodic  KPP equation, for which $c^*$ is the minimal speed of traveling waves \cite{AW-AM78, Fisher, KPP}. In sharp contrast for problem \eqref{main-eq} with the aforementioned shifting nonlinearity, Theorem \ref{ThTW} concludes that $c^*$ is the superemum of the wave speed. 
 
The second result is about the existence, multiplicity and the nonexistence of the forced pulse waves. For this purpose, we need the following additional condition on the asymptotic behavior of $g(t,x,0)$ as $x\to-\infty$: 
\begin{equation}\label{c4}
\text{$\sup_{t\in [0,T]}|g(t,x,0)-g(t,-\infty,0)|=o(|x|^{-r_0-m}) $ for some $r_0>0$ and $m\in \{1,2\}$}.
\end{equation}

\begin{theorem}\label{ThPulse}
The following statements are valid:
\begin{enumerate}
\item[(i)]If $c>-c^*$, then there is no forced pulse wave.
\item[(ii)]If either $c< -c^*$ and \eqref{c4} with $m=1$ holds, or $c=-c^*$ and \eqref{c4} with $m=2$ holds,  then there exist infinitely many forced pulse waves.
\end{enumerate}
\end{theorem}

The third result is about the propagation behavior of the solutions of \eqref{main-eq}.

\begin{theorem}\label{ThSS}
Assume that  $\phi\in C(\R,\R_+)\setminus\{0\}$ is bounded.  Let $u(t,x)$ be the solution of \eqref{main-eq} with $u(0,x)=\phi(x)$. Then the following propagation dynamics holds:
\begin{enumerate}
\item[(i)]If $c\le -c^*$ and $\int_\R e^{-\f{c^*}{2}x}\phi(x)dx<+\infty$, then $\lim_{t\to\infty}\sup_{x\in\R}u(t,x)=0$.
\item[(ii)]If $c\ge c^*$, then 
\[
\lim_{t\to\infty}\sup_{|x|\le \mu t} |u(t,x)-\alpha(t)|=0,\, \forall  \mu\in(0,c^*),\, \, \lim_{t\to\infty}\sup_{|x|\ge c^*t-\gamma \ln t} u(t,x)=0, \, \forall \gamma\in \left(0,\f{2}{c^*}\right).
\]
\item[(iii)]If $c\in (-c^*,c^*)$, then 
\begin{equation}\label{111}
\lim_{t\to\infty}\sup_{x\ge-\mu t} |u(t,x)-U(t,x-ct)|=0,\quad \forall \mu\in (0,c^*),
\end{equation}
where $U(t,x)$ is the unique KPP wave of \eqref{TW}. Further, if $\int_\R e^{-\f{c^*}{2}x}\phi(x)dx<+\infty$ then
\begin{equation}\label{222}
\text{$\lim_{t\to\infty}\sup_{x\le -c^*t+\gamma \ln t} u(t,x)=0, \quad \forall \gamma \in \left(0,\f{2}{c^*}\right)$},
\end{equation}
while if $\liminf_{x\to-\infty}\phi(x)>0$  and there exists $\delta>0$ such that $g_u(t,\xi, u)<-\delta$ for $t\in [0,T],\xi\in\R$ and $u\in [-\delta, \delta+ \max_{t\in[0,T]}\alpha(t)]$, then
\begin{equation}\label{333}
 \lim_{t\to\infty}\sup_{x\in\R} |u(t,x)-U(t,x-ct)|e^{\sigma t}=0\quad \text{for some $\sigma>0$}.
\end{equation}
\end{enumerate}
\end{theorem}

Roughly speaking, if the environment shifts with a moderate speed, then the species eventually propagates like the unique forced time periodic KPP type wave that has the same speed as the environment shifting one; if the environment shifts towards left with large speed, then the good environment shrinks so quickly that the species cannot follow, leading to extinction;  if the environment shifts towards right with large speed, then the good environment expands faster than species propagation. 

When modeling the shifting environment subject to seasonal changes, one may naturally come up with the reaction term $g(t, x-c(t), u)$, where $g(t,x,u)$ and $c'(t)$ are both $T$-periodic in $t$. It means that the environment moves with a periodically fluctuating speed. We remark that such a reaction term can also be casted into the one in system \eqref{main-eq}. Indeed, let $c_0$ be average of $c'(t)$ and $a(t):=c_0t-c(t)$.  Then we have $x-c(t)=x-c_0 t+a(t)$, and 
\[
a(t+T)-a(t)=c_0T+c(t)-c(t+T)=c_0T-\int_0^Tc'(t)dt=0,\quad t\in\R.
\]
Define $f(t,x,u):=g(t,x+a(t),u)$.  It easily follows that $f(t,x,u)$ is
$T$-periodic in $t$, and $g(t, x-c(t), u)=f(t,x-c_0t,u)$.

The ideas for the proof of Theorems \ref{ThTW}-\ref{ThSS} are mainly from dynamical systems and elliptic/parabolic equations, including the property of Poincar\'e maps, super-sub solution method and the sliding argument. Since it is still unclear how the generalized eigenvalues depend on the parameters, we will not directly apply the results of generalized eigenvalues, but some ideas behind their proofs will be useful in the construction of various super- and 
sub-solutions. 

The rest of this paper is organized as follows. In section 2, some preliminary properties are presented, including the a priori estimate on possible forced waves. In sections 3 and 4, the forced KPP type and pulse waves are investigated, respectively. Section 5 is devoted to the study of spreading properties of
solutions to the initial value problem. Finally, a model arising from the pathogen spread among the invasive host is discussed to illustrate the 
obtained analytic results. 

\section{Preliminaries}

We first give some notations that will be  used hereafter. By the definition of $c^*$ in \eqref{def-cstar}, we see that for $|c|\ge c^*$ equation
\begin{equation}\label{ave}
\f{1}{T}\int_0^T [\lambda^2+c\lambda+g(s,-\infty,0)]ds=0
\end{equation}
admits two real solutions. Let  $\lambda_{1,c}$ to be the one with smaller absolute value. More precisely, 
\begin{equation}\label{lambda1}
\lambda_{1,c}:=
\begin{cases}
\f{-c+\sqrt{c^2-(c^*)^2}}{2}, & c\ge c^*\\
\f{-c-\sqrt{c^2-(c^*)^2}}{2}, &c\le-c^*.
\end{cases}
\end{equation}
Clearly, $\lambda_{1,c}>0$ when $c\le -c^*$, and $\lambda_{1,c}<0$ when $c\ge c^*$.

From  (G4) we see that for every $c\in\R$, the equation
\begin{equation}
\f{1}{T}\int_0^T [\mu^2+c\mu+g(s,+\infty,0)]ds=0
\end{equation}
admits a unique negative solution, say $\mu_c$.

For any given $y\in\R \cup\{\pm\infty\}$, let $U^y(t,x; \phi)$ be  the unique solution of
\begin{equation}\label{TW-y}
U_t=U_{xx}+cU_x+Ug(t,x+y, U), \quad t>0
\end{equation} 
satisfying $U(0,x)=\phi(x)$.  If $y=0$, then \eqref{TW-y} reduces to \eqref{TW}. If $y=\pm \infty$, then \eqref{TW-y} reduces to the limiting homogenous equations $U_t=U_{xx}+cU_x+Ug(t,\pm\infty, U)$, respectively.
By the uniqueness of solutions, it easily follows  that 
\begin{equation}\label{translation}
U^y(t,x; \phi(\cdot+y))=U^0(t,x+y; \phi), \quad \forall t\geq 0, x,y\in \R.
\end{equation}

Let $x_1<x_2<\cdots<x_m$ be a finite sequence of real numbers.  A function $\underline{v}\in C^{1,2}(\R\times \R\setminus\{x_1,\cdots, x_m\})$ is said to be a generalized sub-solution of \eqref{TW-y} provided that 
\begin{equation}
\begin{cases}
\underline{v}_t\le \underline{v}_{xx}+c\underline{v}_x+\underline{v}g(t,x+y,\underline{v}),& x\neq x_i,  1\leq i\leq m \\
\underline{v}_x(t,x^+)\ge \underline{v}_x(t,x^-), &  x=x_i,   1\leq i\leq m.
\end{cases}
\end{equation}
A generalized sup-solution can be defined by reversing the above inequalities. 

Let $P_y: C(\R,\R_+) \to C(\R,\R_+)$ be the Poincar\'e map associated with  the time periodic reaction-diffusion equation \eqref{TW-y}, that is, 
$P_y[\phi]:=U^y(T,\cdot;\phi)$. It is easy to verify that $P_y$  admits the 
following properties.

\begin{lemma}\label{P}
The following statements are valid:	
\begin{enumerate}
\item[(i)] If $\phi\ge \psi$, then $P_y[\phi]\ge P_y[\psi]$.
\item[(ii)]If $\phi$ is non-increasing, then $P_y[\phi]$ is also non-increasing.
\item[(iii)] If $\phi_n$ is uniformly bounded, then $P_y[\phi_n]$, up to a subsequence, converges locally uniformly.
\item[(iv)]If $\underline{v}\le \bar{v}$ are a pair of generalized sub- and sup-solutions of \eqref{TW-y} for $t\in (0,t_0)$ with $t_0>T$, then
\begin{equation*}
\underline{v}(0,\cdot)\le P_y[\underline{v}(0,\cdot)]\le P_y[\bar{v}(0,\cdot)] \le \bar{v}(0,\cdot).
\end{equation*}
If, in addition, $\bar{v}(0,\cdot)\ge \bar{v}(T,\cdot)$, then 
$U^y(t,x;\bar{v}(0,\cdot))$  converges,  as $t\to\infty$, to a time periodic solution $U^{y,*}(t,x)$ of \eqref{TW-y} locally uniformly  such that
$\underline{v}\le U^{y,*}\le \bar{v}$. So does $U^y(t,x,\underline{v}(0,\cdot))$ provided that $\underline{v}(0,\cdot)\le \underline{v}(T,\cdot)$.  
\end{enumerate}
\end{lemma}

\begin{lemma}\label{est}
Let $U\not\equiv 0$ be a nonnegative and bounded solution of \eqref{TW}. Then $0<U(t,x)<\alpha(t)$ and $U(t,+\infty)= 0$.
\end{lemma}
\begin{proof}
By the strong maximum principle, we have $U(t,x)>0$. Recall that $\alpha'(t)=\alpha(t)g(t,-\infty, \alpha(t))$. In view of (G2) we conclude that $M\alpha(t)$ is a sup-solution of \eqref{TW} for any $M\ge 1$. Choose $M$ large enough such that $M\alpha(t)\ge U(t,x)$. By Lemma \ref{P} (iv), we then obtain 
\begin{equation*}
P_0^n[M\alpha(0)]\ge P_0^{n+1}[M\alpha(0)]\ge U(0,\cdot),\quad \forall n\ge 0.
\end{equation*}
Note that $P_0^n[M\alpha(0)]$ converges to some $\phi$  locally uniformly.  By Lemma \ref{P} (ii)-(iii), it follows  that $\phi$ is non-increasing and $\phi=P_0[\phi]$. Thus, the equity \eqref{translation} with $t=T$ and $x=0$ implies that
\begin{equation}\label{lim1}
\phi(+\infty)=\lim_{y\to\infty}\phi(y)=\lim_{y\to\infty}P_0[\phi](y)=\lim_{y\to\infty}P_y[\phi(\cdot+y)](0).
\end{equation}
Since $\phi(x+y)$ and $g(t,x+y,u)$ converge to $\phi(+\infty)$ and $g(t,+\infty, u)$ locally uniformly in $(t,x,u)$ as $y\to+\infty$, respectively, 
it easily follows that  $U^y(t,x;\phi(\cdot+y))$ converges to $U^{+\infty}(t,x;\phi(+\infty))$ locally uniformly as $y\to +\infty$. Consequently, we have
\begin{equation}\label{lim2}
\lim_{y\to\infty}P_y[\phi(\cdot+y)](0)=\lim_{y\to\infty} U^y(T,0;\phi(\cdot+y))= U^{+\infty}(T,0;\phi(+\infty))= P_{+\infty}[\phi(+\infty)].
\end{equation}
Combining \eqref{lim1} and \eqref{lim2} yields $\phi(+\infty)=P_{+\infty}[\phi(+\infty)]$. Similarly, $\phi(-\infty)=P_{-\infty}[\phi(-\infty)]$. As such, $\phi(\pm\infty)$ are spatially homogeneous fixed points of $P_{\pm\infty}$, respectively, and hence, $\phi(+\infty)=0$ and $\phi(-\infty)=\alpha(0)$ or $0$. Since $\phi\ge U(0,\cdot)$, we have $\phi(-\infty)=\alpha(0)$, which implies that $U(t,x)\le \alpha(t), \, \forall t\geq 0$.  In view of $\phi(+\infty)=0$  and $\phi\ge U(0,\cdot)$, it then follows from \eqref{translation} that 
$U(t,+\infty)=0$ uniformly for $t\in [0,T]$.  
Finally, by the strong maximum principle again, we obtain $U(t,x)<\alpha(t)$  for all $(t,x)\in\R^2$. 
\end{proof}

\section{Forced KPP type waves}
In this section, we prove Theorem \ref{ThTW} on forced KPP type waves, which 
is a straightforward consequence of the following three lemmas.
\begin{lemma}\label{nonex}
{\sc (Nonexistence)} Problem \eqref{TW} admits no positive solution when $c\ge c^*$.
\end{lemma}
\begin{proof}
Assume, for the sake of contradiction, that $U(t,x)$ is a positive solution.  By Lemma \ref{est} we have $U(t,+\infty)=0$.  We will use a sliding argument to reach a contradiction. 

Recall that $\lambda_{1,c}$ is defined in \eqref{lambda1}. For $M>0$, define 
\begin{equation}
w_1^M(t,x):= Me^{\lambda_{1,c}x}\psi(t),
\end{equation}
where $\psi$ is a $T$-periodic positive function that will be specified later. By (G2) we have 
\begin{equation}
g(t,x,u)\le g(t,-\infty,u)\le g(t,-\infty, 0),\quad t\in\R, x\in\R, u\ge 0.
\end{equation}
It  then follows that  
\begin{eqnarray*}
&& \partial_tw_{1}^M-\partial_{xx}w_{1}^M-c\partial_x w_{1}^M-w_1^Mg(t,x,w_1^M)\\
&&\ge  \partial_t w_{1}^M-\partial_{xx}w_{1}^M-c\partial_x w_{1}^M-w_1^M g(t,-\infty, 0)\\
&&= Me^{\lambda_{1,c}x} \psi(t)\left[ \f{\psi'(t)}{\psi(t)}-\lambda_{1,c}^2-c\lambda_{1,c}-g(t,-\infty,0)\right]\\
&&= 0
\end{eqnarray*}
provided that
\begin{equation}
\f{\psi'(t)}{\psi(t)}=\lambda_{1,c}^2+c\lambda_{1,c}+g(t,-\infty,0),
\end{equation}
which is satisfied if we choose
\begin{equation}
\psi(t)=e^{\int_0^t [ \lambda_{1,c}^2+c\lambda_{1,c}+g(s,-\infty,0)]ds}.
\end{equation}
Clearly, $\psi>0$ is $T$-periodic in view of \eqref{ave}. So $w_1^M$ is a time periodic sup-solution of \eqref{TW}. 

Next we employ an argument used in \cite{HR-JEMS11}. 

{\it Claim 1.  $U(t,x)=o\left(e^{(\mu_c+\eta)x}\right)$ for any $\eta>0$ as $x\to+\infty$. }

Let us postpone the proof of the claim and reach the conclusion quickly. Indeed, from (G2)  we infer that 
\begin{equation}
\mu_c<\lambda_{1,c},
\end{equation}
and hence, there exist $M^*>0$ and $x=x(t)\in \R$ such that
\begin{equation}
w_1^{M^*}(t,x)\ge U(t,x),\quad w_1^{M^*}(t,x(t))=U(t,x(t)).
\end{equation}
Define $W=w_1^{M^*}-U$. Then we have $W\ge 0$, $W(t,x(t))=0$. Further, in view of  $\partial_t w_{1}^{M^*}-\partial_{xx}w_{1}^{M^*}-c\partial_x w_{1}^{M^*}-w_1^{M^*} g(t,-\infty, 0)=0$, \eqref{TW} and (G2), we have
\begin{eqnarray}
&&\partial_t W-\partial_{xx}W-c\partial_xW\nonumber\\
&&= w_1^{M^*} g(t,-\infty,0)-Ug(t,x,U)\nonumber\\
&&\ge  W g(t,-\infty,0).
\end{eqnarray}
By using the parabolic strong maximum principle, we obtain that $w_1^{M^*}(t,x)\equiv U(t,x)$, a contradiction. 

Now we return to the proof of  Claim 1.  In view of (G4), we have
\begin{equation}\label{epsilon0}
\epsilon_0:=-\f{1}{T} \int_0^T g(t,+\infty,0)dt >0.
\end{equation}
For any given $\epsilon\in (0,\epsilon_0)$, there exists $\eta\in (0,-\mu_c)$ such that 
\begin{equation}
\f{1}{T}\int_0^T \left[(\mu_c+\eta)^2+c(\mu_c+\eta)+g(t,+\infty,0)+\epsilon\right]dt=0,
\end{equation}
and we see from  (G1)  that there exists $x_\epsilon \in\R$ such that 
\begin{equation}\label{gplusinf}
g(t,x,0)\le g(t,+\infty,0)+\epsilon, \quad x\ge x_\epsilon.
\end{equation} 
Define the positive periodic function $\psi_\epsilon$ by
\begin{equation*}
\psi_\epsilon(t):=\int_0^t \left[(\mu_c+\eta)^2+c(\mu_c+\eta)+g(s,+\infty,0)+\epsilon\right]ds.
\end{equation*}
For such $x_\epsilon$, there exists $M_\epsilon>0$ such that 
\begin{equation}\label{Mep}
M_\epsilon e^{(\mu_c+\epsilon)x_\epsilon}\psi_\epsilon(t) \ge  \max_{s\in[0,T]}\alpha(s), \quad x\le x_\epsilon, t\in\R.
\end{equation}
Then we can verify that 
\begin{equation*}
w_3^\epsilon (t,x):=\min\{ M_\epsilon e^{(\mu_c+\epsilon)x_\epsilon}\psi_\epsilon(t), \alpha(t)   \}
\end{equation*}
is a time periodic generalized sup-solution of \eqref{TW}. 

It then suffices to show that $w_3^\epsilon (t,x)\ge U(t,x)$. Indeed, define $\Omega:=\{(t,x)\in\R^2: w_3^\epsilon (t,x)\ge \alpha(t)\}$.  It remains to prove that $w_3^\epsilon (t,x)\ge U(t,x)$ for $(t,x)\not\in \Omega$. For this purpose, we try to find a positive periodic function $p_1(t)$ such that for any small $\delta>0$, $w_3^\epsilon (t,x)+\delta p_1(t), (t,x)\not\in \Omega$ is also a super solution of \eqref{TW}. It this were true, then by the sliding method and the parabolic strong maximum principle, we conclude that $w_3^\epsilon (t,x)+\delta p_1(t)> U$ for $\delta>0$, and hence, passing $\delta\to 0$ we obtain $w_3^\epsilon (t,x)\ge  U$. 

As the final step, we choose $p(t)$ such that $w_3^\epsilon (t,x)+\delta p(t), (t,x)\not\in \Omega$ is a super solution of \eqref{TW}.  Indeed, for $(t,x)\not\in \Omega$, by \eqref{Mep} we have $x>x_\epsilon$, and hence, \eqref{gplusinf} holds when $(t,x)\not\in \Omega$. It then follows that
\begin{eqnarray*}
&&\partial_t(w_3^\epsilon +\delta p_1)-\partial_{xx}(w_3^\epsilon +\delta p_1)-c\partial_x (w_3^\epsilon +\delta p_1)-(w_3^\epsilon +\delta p_1)g(t,x,w_3^\epsilon +\delta p_1)\nonumber\\
&&= w_3^\epsilon[g(t,+\infty,0)+\epsilon-g(t,x,w_3^\epsilon+\delta p_1)]+\delta p_1[\f{p_1'}{p_1}-g(t,x,w_3^\epsilon+\delta p_1)]\nonumber\\
&&\ge  \delta p_1[\f{p_1'}{p_1}-g(t,+\infty,0)-\epsilon]\nonumber\\
&&\ge  \delta p_1[\f{p_1'}{p_1}-g(t,+\infty,0)-\epsilon_0]\nonumber\\
&&=0
\end{eqnarray*}
provided that $p_1(t):=e^{\int_0^t [g(s,+\infty,0)+\epsilon_0] ds}$, which is periodic due to \eqref{epsilon0}. 
\end{proof}

\begin{lemma}\label{uniq}
{\sc (Uniqueness)} Problem \eqref{TW} admits at most one KPP type wave when $c< c^*$.
\end{lemma} 
\begin{proof}
Let $p(t,x)$ be a positive function such that 
\begin{equation*}
p(t,x)=p(t+\omega,x), \quad p(t,\pm\infty)>\eta \quad \text{for some $\eta>0$.}
\end{equation*}
Function $p$ will be specified later. Assume that $U_i,i=1,2$ are two positive solutions of \eqref{TW}. For any $\epsilon>0$,  let\footnote{This type of argument was motivated by \cite[Theorem 3.3]{BR-DCDS08}. If the nonlinearity $ug(t,x,u)$ is independent of $t$ and $p$ is chosen to be identically $1$, then we retrieve their proof.}
\begin{equation*}
K_\epsilon:=\{k:  \,  \, kU_1(t,x)\ge U_2(t,x)-\epsilon p(t,x), (t,x)\in\R^2\},
\end{equation*}
which is not empty since 
\begin{equation*}
\text{$\max\left\{0, \f{U_2(t,x)-\epsilon p(t,x)}{U_1(t,x)}\right\}$ is uniformly bounded in $\R^2$.}
\end{equation*}
Define 
\begin{equation}\label{k-epsilon}
k_\epsilon:=\inf K_\epsilon.
\end{equation}
Clearly, $k_\epsilon U_1(t,x)\ge U_2(t,x)-\epsilon p(t,x)$ and $k_\epsilon$ is non-increasing in $\epsilon$. Define 
\begin{equation}
k^*=\lim_{\epsilon\downarrow 0} k_\epsilon.
\end{equation}
It is easy to see that $k^*\in[1,+\infty]$. Next we employ contradiction arguments to exclude the possibility that $k^*\in (1,+\infty]$. 

Define 
\begin{equation}\label{w-epsilon}
w_\epsilon(t,x):=k_\epsilon u_1(t,x)-u_2(t,x)+\epsilon p(t,x)
\end{equation}
and
\begin{equation*}
w(t,x):=\lim_{\epsilon\downarrow 0} w_\epsilon(t,x)=k^*U_1(t,x)-U_2(t,x) \quad \text{when $k^*\in (1,+\infty)$.}
\end{equation*}
Since $k^*$ is assumed, for the sake of contradiction, to be strictly greater than $1$, we may infer that
$k_\epsilon>1$ for sufficiently small $\epsilon$, that
$w_\epsilon(t,\pm\infty)\ge \epsilon p(t,\pm\infty)>\epsilon\eta>0$, 
and that $w(t,-\infty) =(k^*-1)\alpha(t)>0, w(t,+\infty)=0$ when $k^*<+\infty$.
By the definition of $k_\epsilon$ in \eqref{k-epsilon} and $w_\epsilon$ in \eqref{w-epsilon}, we can find $(t_\epsilon,x_\epsilon)$ such that 
\begin{equation}\label{x-epsilon}
w_\epsilon(t_\epsilon,x_\epsilon)=0 \quad \text{and}\quad w_\epsilon(t,x)\not\equiv 0 \quad \text{in any neighborhood of  $(t_\epsilon, x_\epsilon)$.}
\end{equation}
Since $w_\epsilon$ is periodic it $t$, we may assume that $t_\epsilon\in [0,\omega]$. Meanwhile, if $k^*=+\infty$, then we can infer that $x_\epsilon\to+\infty$.

We continue with three different possibilities for $\{x_\epsilon\}$.  (a) $\{x_\epsilon\}$ is bounded. As such, $k^*<+\infty$ and $(t_\epsilon,x_\epsilon)$ converges to (up to subsequence) some point $(t^*,x^*)$, and hence, $w(t^*,x^*)=0$ and
\begin{eqnarray*}
\partial_tw&=&k^* \partial_t U_1-\partial_tU_2\nonumber\\
&=&\partial_{xx}w +c\partial_x w+k^* U_1g(t,x,U_1)-U_2g(t,x,U_2)\nonumber\\
&=&\partial_{xx}w +c\partial_x w+wg(t,x,U_1)+U_2(g(t,x,U_1)-g(t,x,U_2))\nonumber\\
&\ge& \partial_{xx}w +c\partial_x w+wg(t,x,U_1)+U_2(g(t,x,k^*U_1)-g(t,x,U_2))\nonumber\\
&=&\partial_{xx}w +c\partial_x w+wg(t,x,U_1)+U_2\partial_u g(t,x,\xi(t,x))w
\end{eqnarray*}
for some $\xi(t,x)$ in between $k^*U_1(t,x)$ and $U_2(t,x)$. By the strong parabolic maximum principle, we then obtain that $w\equiv 0$, which contradicts the fact that $w(t,-\infty)>0$. Next we consider the possibility (b) $x_\epsilon \to+\infty$ (up to subsequence) with $k^*\in(1,+\infty]$. For small $\epsilon$, $k_\epsilon>1$ and there exists $\rho_\epsilon>0$ such that in the domain 
\begin{equation*}
D_\epsilon:=[0,T]\times (x_\epsilon-\rho_\epsilon, x_\epsilon+\rho_\epsilon),
\end{equation*}
\begin{eqnarray*}
\partial_tw_\epsilon&&=k_\epsilon \partial_t U_1-\partial_tU_2+\epsilon \partial_t p\nonumber\\
&&= \partial_{xx}w_\epsilon +c\partial_x w_\epsilon +k_\epsilon U_1g(t,x,U_1)-U_2g(t,x,U_2)+\epsilon [\partial_t p-D(t)\partial_{xx}p-c\partial_x p] \nonumber\\
&&\ge  \partial_{xx}w_\epsilon +c\partial_x w_\epsilon +(k_\epsilon U_1-U_2)g(t,x,U_1)+U_2[g(t,x,k_\epsilon U_1)-g(t,x,U_2)\nonumber\\
&&\quad +\epsilon [\partial_t p-\partial_{xx}p-c\partial_x p] \nonumber\\
&&= \partial_{xx}w_\epsilon +c\partial_x w_\epsilon+[g(t,x,U_1)+U_2\partial_u g(t,x,\xi(t,x))](k_\epsilon u_1-u_2)\nonumber\\
&&\quad +\epsilon [\partial_t p-\partial_{xx}p-c\partial_x p] \quad\text{for some $\xi(t,x)$ in between $k_\epsilon U_1(t,x)$ and $U_2(t,x)$} \nonumber\\
&&\ge \partial_{xx}w_\epsilon +c\partial_x w_\epsilon+[g(t,x,U_1)+U_2\partial_u g(t,x,\xi(t,x))]w_\epsilon
\end{eqnarray*}
provided that 
\begin{equation}\label{p1}
[\partial_t p-\partial_{xx}p-c\partial_x p] -[g(t,x,U_1)+U_2\partial_u g(t,x,\xi(t,x))]p \ge 0.
\end{equation}
Let us postpone the construction of $p(t,x)$ such that \eqref{p1} holds in $D_\epsilon$. By the strong parabolic maximum principle,  we then obtain $w_\epsilon\equiv 0$ in $D_\epsilon$. This contradicts the choice of $(t_\epsilon, x_\epsilon)$ in \eqref{x-epsilon}.  As for the possibility (c) $x_\epsilon\to-\infty$ (up to subsequence), we have the following inequality in domain $D_\epsilon$ (probably with smaller $\rho_\epsilon$)
\begin{equation*}
\partial_tw_\epsilon \ge \partial_{xx}w_\epsilon +c\partial_x w_\epsilon+[g(t,x,U_1)+U_2\partial_u g(t,x,\eta(t,x))]w_\epsilon
\end{equation*}
for some $\eta(t,x)$ in between $k_\epsilon U_1(t,x)$ and $U_2(t,x)$, provided that \eqref{p1} with $\xi(t,x)$ being replaced by $\eta(t,x)$ holds. By the strong parabolic maximum principle again, we are led to a contradiction. Therefore, $k^*=1$, and hence, $U_1\ge U_2$. The proof is then complete by exchanging the role of $U_1$ and $U_2$. 

Finally, we construct $p(t,x)$ such that \eqref{p1} holds for $t\in \R$ and all sufficiently large $|x|$. Note that $U_i(t,+\infty)=0$ and $U_i(t,-\infty)=\alpha(t)$. It then follows that 
\begin{equation*}
\lim_{x\to+\infty} g(t,x,U_1)+U_2\partial_u g(t,x,\xi(t,x))=g(t,+\infty,0)
\end{equation*}
and
\begin{equation*}
\limsup_{x\to-\infty}g(t,x,U_1)+U_2\partial_u g(t,x,\eta(t,x))=g(t,-\infty,\alpha(t))+\alpha(t)\max_{u\in[\alpha(t),k_\epsilon \alpha(t)]}\partial_u g(t,-\infty, u).
\end{equation*}
Combining with (G2), we see that there exists $x_+>0$ such that 
\begin{equation*}
g(t,x,U_1)+U_2\partial_u g(t,x,\xi(t,x)) \le a_+(t), \quad x\ge x_+,
\end{equation*}
where 
\[
a_+(t):=g(t,+\infty,0)-\f{1}{T}\int_0^T g(t,+\infty,0)dt
\]
has zero average. Similarly,  there exists $x_-<0$ such that 
\begin{equation*}
g(t,x,U_1)+U_2\partial_u g(t,x,\eta(t,x))\le a_-(t), \quad x\le x_-,
\end{equation*}
where 
\begin{eqnarray*}
a_-(t):=&& g(t,-\infty,\alpha(t))+\alpha(t)\max_{u\in[\alpha(t),k_\epsilon \alpha(t)]}\partial_u g(t,-\infty, u)
 \nonumber\\
 && -\f{1}{T}\int_0^T\alpha(t)\max_{u\in[\alpha(t),k_\epsilon \alpha(t)]}\partial_u g(t,-\infty, u)dt
\end{eqnarray*}
also has zero average due to the fact that 
\begin{equation*}
\int_0^T g(t,-\infty,\alpha(t))=\int_0^T \f{\alpha'(t)}{\alpha(t)}=0.
\end{equation*}
Choose $p\in C^{1,2}(\R\times\R)$ such that $p$ is periodic in $t$, 
$p(t,x)=e^{\int_0^t a_-(s)ds}, \, \forall x<x_-$,
and  $p(t,x)=e^{\int_0^t a_+(s)ds},  \, \forall x>x_+$.
Then \eqref{p1} holds for all $t\in[0,T]$ and all sufficiently large $|x|$. 
\end{proof}

\begin{lemma}\label{ex1}
{\sc (Existence)} Problem \eqref{TW} admits a KPP type wave when $c<c^*$.
\end{lemma} 
\begin{proof}
From the proof of Lemma \ref{est}, we see that $P_0^n[\alpha(0)]$ converges to a non-increasing function $\phi$ such that 
\begin{equation*}
P_0[\phi]=\phi,\quad \phi(+\infty)=0,\quad \text{$\phi(-\infty)=0$ or $\alpha(0)$.}
\end{equation*} 
Note that if $\phi(+\infty)=\alpha(0)$, then the time periodic extension of $U(t,x;\phi)$ is a KPP wave of \eqref{TW}. 

It suffices to show that $\phi\not\equiv 0$. Indeed, we only need to construct appropriate generalized sub-solutions $\underline{v}$ with $\underline{v}(t,x)\le \alpha(t)$ for all $t\in[0,2T]$ and $x\in\R$. 

We proceed with two cases: (i) $c\in (-c^*,c^*)$; (ii) $c\le -c^*$.

Case (i). $c\in (-c^*,c^*)$. We claim that
\begin{equation*}\label{sub-com}
\underline{v}(t,x):=
\begin{cases}
\delta p(t) e^{\lambda t-\f{c}{2}x}\sin\f{\pi(x+M)}{L}, & x\in [-M,-M+L]\\
0, & otherwise
\end{cases}
\end{equation*}
is a generalized sub solution of \eqref{TW} for $t\in (0, 2T)$ provided that the positive constants $\delta, \lambda, M, L$ and the periodic function $p(t)$ are appropriately chosen.  

Indeed, it suffices to find $M>0$ and $L>0$ such that that for $x\in(-M, -M+L)$ and $t\in (0, 2T)$ there holds $-\underline{v}_t+\underline{v}_{xx}+c\underline{v}_x+\underline{v}g(t,x,\underline{v})\ge 0$. Indeed, for $\epsilon>0$, we see from (G1) that there exist $\delta_\epsilon>0$ and $x_\epsilon<0$ such that
\begin{equation*}
g(t,x,u)\ge g(t,-\infty,0)-\epsilon, \quad u\in [0,\delta_\epsilon], x\le x_\epsilon.
\end{equation*}
By direct computations, we obtain
\begin{equation*}
\underline{v}_t=\left(\lambda +\f{p'}{p}\right) \underline{v},
\end{equation*}
and
\begin{eqnarray*}
&&\underline{v}_{xx}+c\underline{v}_x\nonumber\\
&&=\left\{\left[\f{c^2}{4}-\left(\f{\pi}{L}\right)^2-\f{c\pi}{L}\cot\f{\pi(x+M)}{L}\right]+c\left(-\f{c}{2\theta}+\f{\pi}{L}\cot\f{\pi(x+M)}{L}\right)\right\}\underline{v}\nonumber\\
&&= \left[-\f{c^2}{4}-\left(\f{\pi}{L}\right)^2\right]\underline{v}.
\end{eqnarray*}
Since $c\in(-c^*, c^*)$, there exists $L_0>0$ and $\epsilon_0>0$ such that 
\begin{equation}\label{ineq11}
\f{c^2}{4}+\left(\f{\pi}{L}\right)^2<\f{(c^*)^2}{4}-\epsilon,\quad L\ge L_0, \epsilon\le \epsilon_0.
\end{equation}
Consequently, for all $x\in(x_\epsilon-L, x_\epsilon)$  and $t\in (0,2T)$ we arrive at
\begin{eqnarray*}
&& \underline{v}^{-1}[-\underline{v}_t+\underline{v}_{xx}+c\underline{v}_x+\underline{v}g(t,x,\underline{v})]\nonumber\\
 &&=-\lambda-\f{p'}{p}-\f{c^2}{4}-\left(\f{\pi}{L}\right)^2+g(t,x,\underline{v})\nonumber\\
&&\ge  -\lambda-\f{p'}{p}-\f{c^2}{4}-\left(\f{\pi}{L}\right)^2+g(t,-\infty,0)-\epsilon\nonumber\\
&&=0
\end{eqnarray*}
provided that
\begin{equation}\label{ineq12}
\sup_{t\in(0,2T), x\in(x_\epsilon-L,x_\epsilon)}\underline{v}(t,x)\le \delta_\epsilon,
\end{equation}
\begin{equation}\label{lambda}
\lambda=\f{(c^*)^2}{4}-\epsilon-\f{c^2}{4}-\left(\f{\pi}{L}\right)^2,
\end{equation}
and 
\begin{equation}\label{22}
p(t)=e^{\int_0^t [-\f{(c^*)^2}{4}+g(s,-\infty,0)]ds}.
\end{equation}
It  then follows from \eqref{ineq11} that $\lambda>0$ for all $\epsilon\in (0,\epsilon_0]$ and $L\ge L_0$. By \eqref{def-cstar},  we further infer that $p(t)=p(t+T)>0$.  
Since 
\begin{equation*}
\sup_{t\in(0,2T), x\in(x_{\epsilon}-L,x_{\epsilon})}\underline{v}(t,x)\le \delta e^{2T\lambda-\f{c}{2}(x_{\epsilon}-L)}\max_{t\in[0,T]}p(t),
\end{equation*}
we see that \eqref{ineq12} holds provided that 
\begin{equation}\label{delta}
\delta\le \left[e^{2T\lambda-\f{c}{2}(x_{\epsilon}-L)}\max_{t\in[0,T]}p(t)\right]^{-1}\delta_{\epsilon}.
\end{equation}
Thus, the  claim is proved by choosing $\lambda$ and $\delta$ as in  \eqref{lambda} and \eqref{delta}, respectively,  with 
\begin{equation*}\label{23}
\epsilon=\epsilon_0, \quad L=L_0,\quad M\ge-x_{\epsilon_0}.
\end{equation*}

Case (ii). $c\le -c^*$. For any $\epsilon>0$, there exists $x_\epsilon<0$ and $\delta_\epsilon$ such that $g(t,x,u)\ge g(t,-\infty,0)-\epsilon, x\le x_\epsilon, u\in [0,\delta_\epsilon]$. We assume $\underline{v}$ has the following form
\begin{equation*}\label{vsub}
\underline{v}(t,x)=\max\{0, \delta p(t) (e^{\mu x}-M e^{(\mu+\eta) x})\},
\end{equation*}
where the positive numbers $\delta, \mu,\eta, M$ and the periodic positive function $p$ will be specified later so that 
\begin{eqnarray}\label{ineq13}
&&-\underline{v}_t+\underline{v}_{xx}+c\underline{v}_x+\underline{v}g(t,x,\underline{v})\nonumber\\
&&\ge -\underline{v}_t+\underline{v}_{xx}+c\underline{v}_x+\underline{v}[g(t,-\infty,0)-\epsilon]\nonumber\\
&&\ge 0
\end{eqnarray}
in the region $\Omega$ where $\underline{v}>0$. To ensure the first inequality in \eqref{ineq13}, it suffices to show that
\begin{equation}\label{ineq17}
(t,x)\in \Omega\Rightarrow x\le x_\epsilon, \quad \underline{v}\le \delta_\epsilon.
\end{equation}
In region $\Omega$, $\underline{v}(t,x)=w(t,x):=\delta p(t) (e^{\mu x}-M e^{\mu+\eta x})$. To ensure the second inequality in \eqref{ineq13}, it then suffices to verify that $-w_t+w_{xx}+cw_{x}+w[g(t,-\infty,0)-\epsilon]\ge 0$ in $\Omega$.  By direct computations, one can quickly obtain
\begin{eqnarray*}
&&-w_t+w_{xx}+cw_{x}+w[g(t,-\infty,0)-\epsilon]\nonumber\\
&&= \delta p(t)e^{\mu x}\left[-\f{p'}{p}+\mu^2+c\mu+g(t,-\infty,0)-\epsilon\right]\nonumber\\
&&\quad -\delta p(t)Me^{(\mu+\eta) x}\left[-\f{p'}{p}+(\mu+\eta)^2+c(\mu+\eta)+g(t,-\infty,0)-\epsilon\right] \nonumber\\
&&\ge 0
\end{eqnarray*}
provided that
\begin{equation}\label{ineq14}
\mu^2+c\mu+\f{1}{T}\int_0^T g(t,-\infty,0)dt-\epsilon=0,
\end{equation}
\begin{equation}\label{ineq15}
(\mu+\eta)^2+c(\mu+\eta)+\f{1}{T}\int_0^T g(t,-\infty,0)dt-\epsilon<0,
\end{equation}
and 
\begin{equation}\label{ineq16}
p(t)=e^{\int_0^t [\mu^2+c\mu+g(s,-\infty,0)-\epsilon]ds}.
\end{equation}
Note that
\begin{equation*}
c^2\ge (c^*)^2=4\f{1}{T}\int_0^T g(t,-\infty,0)dt >4\left[\f{1}{T}\int_0^T g(t,-\infty,0)dt -\epsilon\right].
\end{equation*}
It then follows that there exists $\mu$ and $\eta$ such that \eqref{ineq14} and \eqref{ineq15} hold. As such, $p(t)$ defined in \eqref{ineq16} is periodic. 

To make sure \eqref{ineq17} holds, we may choose $M>0$ and $\delta>0$ such that 
\begin{equation}\label{ineq18}
\f{1}{\eta}\ln \f{1}{M}<x_\epsilon
\end{equation}
and 
\begin{equation}\label{ineq19}
\sup_{x\le x_\epsilon, t\in [0,T]} \underline{v}(t,x)\le \delta \max_{t\in[0,T]}p(t)e^{\mu x_\epsilon}<\delta_\epsilon.
\end{equation}
Therefore, $\underline{v}$ is a generalized sub solution provided that $\mu, \eta, p(t), M$ and $\delta$ are chosen such that \eqref{ineq14}-\eqref{ineq19}  hold.
\end{proof}

\section{Forced pulse waves}
In this section, we prove Theorem \ref{ThPulse} on forced pulse waves, which is a straightforward consequence of the following two lemmas.

\begin{lemma}\label{nopulse}
{\sc (Nonexistence)} Problem \eqref{TW} admits no pulse wave when $c>-c^*$. 
\end{lemma}
\begin{proof}
By Lemma \ref{nonex} we see that there is no pulse wave when $c\ge c^*$. Next we consider $c\in (-c^*,c^*)$. As such, the compact supported function $\underline{v}$ defined in \eqref{sub-com} is a sub-solution provided that \eqref{lambda}, \eqref{22}, \eqref{delta} and \eqref{23} hold. 

We use a sliding argument. Assume, for the sake of contradiction, that there is a pulse wave $U(t,x)$. Then we may adjust parameters $M$ and $\delta$ satisfying \eqref{delta} and \eqref{23} such that $U(t,x)\ge \underline{v}$ for all $(t,x)$ and $U(0,x)=\underline{v}(0,x)$ for some $x=x^*$. Then an application of the strong maximum principle leads to a contradiction. 
\end{proof}

\begin{lemma}\label{exIII}
{\sc (Existence)} Problem \eqref{TW} admits (infinitely many) pulse waves in either of the following two cases: (i) $c< -c^*$ and \eqref{c4} with $m=1$ holds; (2) $c= -c^*$ and \eqref{c4} with $m=2$ holds. 
\end{lemma}
\begin{proof}
In view of $c\le -c^*$, we use $\mu_1$ to denote the smallest positive solution of 
\begin{equation}\label{mu1}
\mu^2+c\mu+\f{1}{T}\int_0^Tg(t,-\infty,0)dt=0
\end{equation} 
Define the periodic function $p(t)$ by
\begin{equation}\label{def-ppp}
p(t):=e^{\int_0^t [g(s,-\infty,0)-\f{1}{T}\int_0^Tg(\theta,-\infty,0)d\theta]ds}.
\end{equation}
By (G2), we  see that 
\begin{equation*}
\bar{v}(t,x):=\delta p(t)e^{\mu_1 x}, \quad \delta>0
\end{equation*} is a sup-solution of \eqref{TW}. 

{\it Claim 2. The function defined by 
\begin{equation*}
\underline{v}(t,x):=
\begin{cases}
\delta p(t)e^{\mu_1 x}[|x|^k-M|x|^{k-r}], & x<-M^{\f{1}{r}}\\
0, &x\ge-M^{\f{1}{r}}
\end{cases}
\end{equation*}
is a generalized sub solution of \eqref{TW} if the positive constants $k,\delta, M$ and $r$ are appropriately chosen. }

Let us postpone the proof of the claim and complete the proof in a few lines. By the same iteration argument as in Lemma \ref{ex1}, we obtain a time periodic positive solution $U^*(t,x)$ of \eqref{TW} with 
\begin{equation*}
\underline{v}(t,x)\le U^*(t,x)\le \bar{v}(t,x), \quad t\in[0,T], x\in\R.
\end{equation*}
By Lemma \ref{ex1},  we obtain $U^*(t,-\infty)=0$. Since $\bar{v}(t,+\infty)=0$, we have $U^*(t,+\infty)=0$.

Now we return to the proof of Claim 2 above. It suffices to check that for any $t\in\R$ and $x<-M^{\f{1}{r}}$, $\mc{L}\underline{v}:=-\underline{v}_t+\underline{v}_{xx}+c\underline{v}_x+\underline{v}g(t,x,\underline{v})\ge 0$. 
By computations,  we obtain
\begin{equation*}
\underline{v}_x=\delta p(t) e^{\mu_1x}[\mu_1(|x|^k-M|x|^{k-r})+x^{-1}(k|x|^k-M(k-r)|x|^{k-r})],
\end{equation*}
\begin{eqnarray*}
\underline{v}_{xx}=&&\delta p(t) e^{\mu_1x} [\mu_1^2(|x|^k-M|x|^{k-r})+2\mu_1  x^{-1}(k|x|^k-M(k-r)|x|^{k-r})\nonumber\\
&& +|x|^{-2}(k(k-1)|x|^{k}-M(k-r)(k-r-1)|x|^{k-r})],
\end{eqnarray*}
and
\begin{eqnarray*}
&&|g(t,x,\underline{v}(t,x))-g(t,-\infty,0)|\nonumber\\
&&\le |g(t,x,\underline{v}(t,x))-g(t,x,0)|+|g(t,x,0)-g(t,-\infty,0)|\nonumber\\
&&\le |\partial_u g(t,x,\xi(t,x))|\underline{v}(t,x)+|g(t,x,0)-g(t,-\infty,0)|\nonumber\\
&&\le  \max_{t,y,s}\delta p(t)|\partial_u g(t,y,s)| |x|^k e^{\mu_1 x}+|g(t,x,0)-g(t,-\infty,0)|,
\end{eqnarray*}
where $\xi(t,x)\in (0,\underline{v}(t,x))$. By assumption \eqref{c4}, we have $g(t,x,0)-g(t,-\infty,0)=o(|x|^{-r_0-m})$ as $x\to-\infty$. Thus, 
\begin{equation*}
g(t,x,\underline{v}(t,x))= g(t,-\infty,0)+o(|x|^{-r_0-m}), \quad \text{as $x\to-\infty$}.
\end{equation*}
It then follows that
\begin{eqnarray}\label{estimate}
&&[\delta p(t) |x|^ke^{\mu_1x}]^{-1}\mc{L}\underline{v}\nonumber\\
&&\ge -\f{p'}{p}(1-M|x|^{-r})+[\mu_1^2(1-M|x|^{-r})+2\mu_1  x^{-1}(k-M(k-r)|x|^{-r})\nonumber\\
&& \quad +|x|^{-2}(k(k-1)-M(k-r)(k-r-1)|x|^{-r})]\nonumber\\
&&\quad +c[\mu_1(1-M|x|^{-r})+x^{-1}(k-M(k-r)|x|^{-r})]\nonumber\\
&&\quad +[g(t,-\infty,0)+o(|x|^{-r_0-m})](1-M|x|^{-r}).
\end{eqnarray}
In view of \eqref{mu1} and \eqref{def-ppp}, we have
\begin{eqnarray*}\label{31}
&&[\delta p(t) |x|^ke^{\mu_1x}]^{-1}\mc{L}\underline{v}\nonumber\\
&&\ge (2\mu_1+c)  x^{-1}(k-M(k-r)|x|^{-r})+|x|^{-2}(k(k-1)-M(k-r)(k-r-1)|x|^{-r})\nonumber\\
&&\quad +o(|x|^{-r_0-m}).
\end{eqnarray*}
Next we proceed with two cases. 

(i) $c<-c^*$. In this case, we choose $k=0$ and assume that \eqref{c4} with $m=1$ holds. Hence, \eqref{estimate} reduces to 
\begin{equation*}
[\delta p(t) |x|^ke^{\mu_1x}]^{-1}\mc{L}\underline{v}\ge  Mr |x|^{-r}x^{-1}[(2\mu_1+c)-x^{-1}(r+1)+o(|x|^{-r_0+r})].
\end{equation*}
This implies that  $\mc{L}\underline{v}\ge0$ provided that 
\begin{equation*}
(2\mu_1+c)-x^{-1}(r+1)+o(|x|^{-r_0-r})\le 0, \quad x<-M^{\f{1}{r}},
\end{equation*}
which is true if $r\in (0,r_0)$ and $M=M(r)$ is sufficiently large, thanks to 
$2\mu_1+c<0$.

(ii) $c=-c^*$. As such, we have $2\mu_1+c=0$.
We then choose $k=1$ and assume that \eqref{c4} with $m=2$ holds. Hence,  \eqref{estimate} reduces to 
\begin{equation*}
[\delta p(t) |x|^ke^{\mu_1x}]^{-1}\mc{L}\underline{v}\ge  |x|^{-r-2}[Mr(1-r) +o(|x|^{-r_0+r})], \quad x<-M^{\f{1}{r}},
\end{equation*}
which is nonnegative provided that $r\in (0,\min\{1,r_0\})$ and $M=M(r)$ is sufficiently large. 

Till now, we have shown the existence of pulse waves when $c\le -c^*$ subject to the condition \eqref{c4}. As for the multiplicity, we shift the above pair of sup-and sub-solutions to the left with a certain length such that the new sub-solution intersects with the established pulse wave. With this new pair of sup- and sub-solutions,  we obtain a different pulse wave. This argumet can be repeated infinitely many times. 
\end{proof}

\section{Propagation dynamics}

In this section, we study various spreading properties of solutions to the 
nonautonomous evolution equation  \eqref{main-eq}.

For any given  $\phi\in C(\R,\R_+)\setminus \{0\}$,  let $u(t,x;\phi)$ be the unique solution of \eqref{main-eq}.  From  Lemmas \ref{uniq} and \ref{ex1}, we see  that for any $c<c^*$, there is a unique KPP type wave $U(t,x-ct)$. Further, the proof of Lemma \ref{est} implies that for any $M\ge \alpha(0)$, the solution $u(t,x;M)$ converges to $U(t,x-ct)$ locally uniformly in the variable $\xi:=x-ct$ as $t\to \infty$.  We further have the following strong convergence
result.

\begin{lemma}\label{uniformconv}
Assume that $c<c^*$ and $M\ge \alpha(0)$. Then $\lim_{t\to\infty}|u(t,x;M)-U(t,x-ct)|=0$ uniformly in $x\in\R$.
\end{lemma}
\begin{proof}
In view of the monotonicity of $U(t,x)$ in $x\in\R$, we have $U(0,x)\le U(0,-\infty)=\alpha(0)$. It then follows from the comparison principle that 
\begin{equation}\label{31}
U(t,x)\le u(t,x;M)
\end{equation} 
due to $M\ge \alpha(0)$. On the other hand, the solution $v(t)$ of
\begin{equation*}
v'(t)=vg(t,-\infty,v), \quad v(0)=M
\end{equation*}
converges to $\alpha(t)$ in the sense that $\lim_{t\to\infty}|v(t)-\alpha(t)|=0$. Thanks to $g(t,x,s)\le g(t,-\infty,0)$, we obtain $u(t,x;M)\le v(t)$ by using again the comparison principle. Consequently, 
\begin{equation}\label{32}
\limsup_{t\to\infty}[u(t,x;M)-\alpha(t)]\le 0.
\end{equation}
Combining \eqref{31}, \eqref{32}, $U(t,-\infty)=\alpha(t)$ and the established local convergence of $u(t,x;M)$ to $U(t,x)$ in the variable $x-ct$, we obtain that for any $x_0\in\R$, 
\begin{equation}\label{33}
\limsup_{t\to\infty}|u(t,x;M)-U(t,x-ct)|=0 
\end{equation}
uniformly in $x-ct\le x_0$. 

Next we show that \eqref{33} holds uniformly in $x-ct\ge x_0$. Assume, for the sake of contradiction, that there exists $\delta_0>0$ and $t_n, x_n$ such that $x_n-ct_n\to+\infty$ and $u(t_n,x_n;M)=\delta_0$, where we have used the fact $U(t,+\infty)=0$. Let $[t_n/T]$ be the integer part of $t_n/T$.  Without loss of generality, we assume that $\lim_{n\to\infty} t_n-[t_n/T]T=t^*$. Define $w_n(t,x):=u(t+t_n,x+x_n;M)$. It then follows that $w_n$ converges locally uniformly in $(t,x)$ to some $w$, which is a nonnegative solution of
\begin{equation*}
w_t=w_{xx}+wg(t+t^*,+\infty, w).
\end{equation*}
Since $g(t,+\infty, s)$ is non-increasing in $s$ and periodic in $t$ and $\int_0^T g(t,-\infty,0)dt<0$, we can conclude that $w\equiv 0$. This leads to a contradiction to 
\[ 
w(0,0)=\lim_{n\to\infty}w_n(0,0)=\delta_0>0.
\]
Thus, \eqref{33} holds uniformly in $x\in \R$.
\end{proof}

The above lemma shows that the solutions of \eqref{main-eq} with ``large" initial datum converge to the forced KPP type wave uniformly when it exists.  Next we consider some kinds of ``small"  initial datum. 

\begin{lemma}\label{conv0}
Assume that $c\le -c^*$. If $|\phi|_{\infty}\le M$ for some $M>0$ and 
\begin{equation*}
\int_\R e^{-\f{c^*}{2}x}\phi(x)dx<+\infty,
\end{equation*} 
then  $\lim_{t\to\infty} u(t,x;\phi)=0$ uniformly in $x\in\R$.
\end{lemma}

As a remark on Lemma \ref{conv0}, we point out  that if the initial function $\phi$ has an exponentially decaying tail that is bigger than $xe^{-\f{c^*}{2}x}$, then $u(t,x;\phi)$ may converge to one of the pulse waves, which is stable in a certain sense. We refer to the construction of pulse waves in Lemma \ref{exIII}. 

\begin{proof}
By Lemma \ref{uniformconv}, we have 
\begin{equation}\label{34}
\limsup_{t\to\infty} [u(t,x;\phi)-U(t,x-ct)]\le 0,\quad \text{uniformly in $x\in\R$}.
\end{equation}
We use an argument in \cite{Bramson-83} to study the solution $v(t,x;\phi)$ of the following linear equation
\begin{equation*}
v_t=v_{xx}+vg(t,-\infty,0),\quad v(0,x)=\phi(x).
\end{equation*}
Thanks to $g(t,-\infty,0)\ge g(t,-\infty, s)\ge g(t,x,s)$, we obtain $u(t,x;\phi)\le v(t,x;\phi), \forall t\ge 0, x\in\R$. Note that $v(t,x)$ can be expressed explicitly as 
\begin{equation*}
v(t,x)=\f{e^{\int_0^t  g(s,-\infty,0)ds}}{\sqrt{4\pi t}}\int_\R e^{-\f{(x-y)^2}{4t}}\phi(y)dy.
\end{equation*}
Let $\sigma(t)$ be a function such that $\lim_{t\to\infty}\sigma(t)=+\infty$. It will be specified later. By direct computations, we then have
\begin{eqnarray*}
v(t,x-c^*t+\sigma(t)) &=& \f{e^{\int_0^t  g(s,-\infty,0)ds}}{\sqrt{4\pi t}}\int_\R e^{-\f{(x-c^*t+\sigma(t)-y)^2}{4t}}\phi(y)dy\nonumber\\
&=& \f{e^{\int_0^t  g(s,-\infty,0)ds}}{\sqrt{4\pi t}}\int_\R e^{-\f{(x+\sigma(t)-y)^2-2c^*t(x+\sigma(t)-y)+(c^*t)^2}{4t}}\phi(y)dy\nonumber\\
&\le & \f{e^{\int_0^t  g(s,-\infty,0)ds}}{\sqrt{4\pi t}}\int_\R e^{-\f{-2c^*t(x+\sigma(t)-y)+(c^*t)^2}{4t}}\phi(y)dy\nonumber\\
&= & \f{e^{\int_0^t  g(s,-\infty,0)ds-\f{(c^*)^2}{4}t+\f{c^*}{2}\sigma(t)}}{\sqrt{4\pi t}}e^{\f{c^*}{2}x}\int_\R e^{-\f{c^*}{2}y}\phi(y)dy.
\end{eqnarray*}
Recall that $c^*=2\sqrt{\f{1}{T}\int_0^T g(t,-\infty,0)dt}$. Consequently, $p(t):=e^{\int_0^t  g(s,-\infty,0)ds-\f{(c^*)^2}{4}t}$ is a positive periodic function. By the assumption that $\int_\R e^{-\f{c^*}{2}y}\phi(y)dy <+\infty$,
it follows  that there exists a constant $C>0$ such that 
\begin{equation*}
\limsup_{t\to\infty}v(t,x-c^*t+\sigma(t)) \le C e^{\f{c^*}{2}x},\quad \forall x\in\R
\end{equation*}
provided that $\f{e^{\f{c^*}{2}}\sigma(t)}{4\pi t}$ is uniformly bounded. For this purpose, we choose
\begin{equation*}
\sigma(t)=\f{2}{c^*}\ln t+C_1,
\end{equation*}
where $C_1$ is a constant. As such, we obtain 
\begin{equation}\label{35}
\limsup_{t\to\infty} \left[u(t,x;\phi)-Ce^{\f{c^*}{2}(x+c^*t-\sigma(t))}\right]\le 0,\quad \text{uniformly in $x\in\R$}.
\end{equation}

Combining \eqref{34} and \eqref{35}, we can infer the conclusion. Indeed, for any $\epsilon>0$ there exists $\xi_0=\xi_0(\epsilon)$ and $t_0=t_0(\xi_0,\epsilon)$ such that 
\begin{equation*}
U(t,x-ct)\le \epsilon, \quad \forall x-ct\ge\xi_0,
\end{equation*} 
and
\begin{equation*}
Ce^{\f{c^*}{2}(\xi_0-\sigma(t))}\le \epsilon, \quad \forall t\ge t_0
\end{equation*}
thanks to $\lim_{t\to\infty}\sigma(t)=+\infty$. Therefore, in view of $c\le -c^*$, we have
\begin{equation*}
\min\{U(t,x-ct), Ce^{\f{c^*}{2}(\xi_0-\sigma(t))}\}\le \epsilon, \quad \forall t\ge t_0.
\end{equation*}
This completes the proof.
\end{proof}

\begin{remark}\label{remark1}
Similar to \eqref{35}, there exists a constant $C>0$ such that
$\limsup_{t\to\infty} [u(t,x;\phi)-Ce^{-\f{c^*}{2}(x-c^*t+\sigma(t))}]\le 0$ uniformly in $x\in\R$  
provided that $\int_\R e^{\f{c^*}{2}y}\phi(y)dy<+\infty$.
\end{remark}

\begin{lemma}\label{moderatespeed}
Assume that $c\in (-c^*, c^*)$. Then $\lim_{t\to\infty}\sup_{x\ge -\mu t}|u(t,x;\phi)-U(t,x-ct)|=0$ for any $\mu\in (0,c^*)$.
\end{lemma}
\begin{proof}
Without loss of generality, we assume that $\epsilon\in(0,|c-c^*|)$ is fixed. By definition, it suffices to show that for any $\eta>0$ there exists $t_0>0$ such that 
\begin{equation*}
\sup_{x\ge -(c^*-\epsilon)t} |u(t,x;\phi)-U(t,x-ct)|<\eta,\quad t\ge t_0.
\end{equation*}
We proceed with two regions: (i) $x\ge -x_0+ct$; (ii) $x\in [-(c^*-\epsilon)t, -x_0+ct]$, where $x_0$ is specified below. Since $U(t,-\infty)=\alpha(t)$ and $\alpha'(t)=\alpha(t) g(t,-\infty, \alpha(t))$, it follows that for any $\eta>0$, there exists $x_0>0$ such that
\begin{equation*}
U(t,x)>\alpha(t)-\eta/2,\quad \forall t\in\R, x\le -x_0
\end{equation*}
and
\begin{equation*}
c_0:=2\sqrt{\f{1}{T}\int_0^T g(t,-x_0,0)dt }>c^*-\epsilon/2.
\end{equation*}
By \cite[Theorem 5.2.1]{ZhaoBook}, the ordinary differential equation
\begin{equation*}
v'(t)=vg(t,-x_0,v)
\end{equation*} 
admits a unique positive periodic solution $\alpha_0(t)$ with 
$\alpha_0(t)>\alpha(t)-\eta,\, \forall t\in\R$.

(i) $x\ge -x_0+ct$. By the similar arguments to the proof of Lemma \ref{uniformconv}, we obtain $u(t,x;\phi)$ converges to $U(t,x-ct)$ uniformly in the variable $\xi:=x-ct\ge -x_0$. Therefore, such $t_0$ can be found.

(ii) $x\in [-(c^*-\epsilon)t, -x_0+ct].$  In this region, it is clear that $U(t,x-ct)\ge U(t,-x_0)\ge \alpha(t)-\eta/2$ for all $t\ge 0$. It then suffices to find $t_0>0$ such that 
\[
\sup_{x\in [-(c^*-\epsilon)t, -x_0+ct]} |u(t,x;\phi)-\alpha(t)|<\eta/2,\, \,   \forall t\ge t_0. 
\]
Indeed, since $x\le -x_0+ct$ we have $g(t,x-ct,s)\ge g(t,-x_0,s)$. Consequently, the comparison principle implies that $u(t,x;\phi)\ge w(t,x;\phi)$ for all $x\in  [-(c^*-\epsilon)t, -x_0+ct] $, where $w$ solves
\begin{equation*}
w_t=w_{xx}+wg(t,-x_0,w),\quad w(0,x)=\phi(x).
\end{equation*}
It is known that the spreading speed of $w(t,x;\phi)$ is $c_0$. In particular, 
\begin{equation*}
\lim_{t\to\infty}\sup_{|x|\le \mu t}|w(t,x;\phi)-\alpha_0(t)|=0, \mu\in (0,c_0).
\end{equation*}
Choose $\mu=c^*-\epsilon$. Then  we have $\mu<c_0$, and hence, there exists $t_0$ such that 
\[
\alpha(t)+\eta/2\ge u(t,x;\max_{x\in\R}\phi(x))\ge u(t,x;\phi)\ge w(t,x;\phi)\ge \alpha(t)-\eta/2, \, \, \forall t\ge t_0. 
\]
This completes the proof.
\end{proof}

\begin{lemma}\label{largespeed}
Assume that $c\ge c^*$. Then $\lim_{t\to\infty}\sup_{|x|\le \mu t}|u(t,x;\phi)-\alpha(t)|=0$ for any $\mu\in (0,c^*)$. If, in addition, $\int_\R e^{-\f{c^*}{2}x}\phi(x)dx<+\infty$, then $\lim_{t\to\infty}\sup_{|x|\ge c^* t-\mu \ln t}u(t,x;\phi)=0$ for any $\mu<\f{2}{c^*}$.
\end{lemma}
\begin{proof}
The first part can be proved by the comparison arguments similar to the proof of Lemma \ref{moderatespeed} (ii). The second part follows from \eqref{35} and Remark \ref{remark1}. Here we omit the details. 
\end{proof}

\begin{lemma}\label{exponential}
Assume that there exists $\delta>0$ such that $g_u(t,\xi, u)<-\delta$ for all $t\in [0,T],\xi\in\R$ and $u\in [-\delta, \delta+ \max_{t\in[0,T]}\alpha(t)]$. Define
 \begin{equation}
 w^{\pm}(t,x):=(1\pm \rho e^{-\sigma t})U(t,x-ct)\pm \sigma\rho e^{-\sigma t} p(t,x-ct),
 \end{equation} 
where $p(t,x)$ is a positive time periodic function. Then $w^\pm$ is a pair of sup- and sub-solutions of \eqref{main-eq} if positive numbers $\rho, \sigma$ and $p(t,x)$ are appropriately chosen. 
\end{lemma}
\begin{proof}

Let $\xi_0$ satisfy $\int_0^Tg(t,\xi_0,0)dt=\f{1}{2}\int_0^T g(t,+\infty,0)<0$. Define $\sigma_0:=-\f{1}{4}\int_0^Tg(t,\xi_0+1,0)dt>0$. Now we are ready to construct $p$. Let $v_0(t)$ be the unique positive periodic solution of $v'(t)=[\sigma_0+g(t,\xi_0+1,0)]v(t)+U(t,\xi_0+1)$. Define $\sigma_1:=\min\{1, \sigma_0, \f{\delta}{2}\min_{t\in[0,T]}U(t,\xi_0)\}$. Let $v_1(t)$ be the unique positive periodic solution of $v'(t)=[\sigma_1+g(t,-\infty,0)]v(t)+\f{\sigma_1-\delta U(t,\xi_0)}{\sigma_1}U(t,\xi_0)$. Let $p\in C^2$ be a positive function such that $p(t+T,x)=p(t,x)$ and $p(t,x)\equiv v_0(t)$ for  $x\ge \xi_0+1$ and $p(t,x)\equiv v_1(t)$ for $x\le \xi_0$.  Choose $\rho_0$ such that $\rho_0(|U|_{\infty}+|p|_\infty)<\delta$. It then follows that  for any $\rho\le \rho_0$,
\begin{eqnarray*}
\mc{L}w^+&&:=-w^+_t+w^+_{xx}+w^+g(t,x-ct,w^+)\nonumber\\
&&=(1+\rho e^{-\sigma t})U[g(t,\xi, w^+)-g(t,\xi,U)]\nonumber\\
&&\quad +\sigma \rho e^{-\sigma t}[-p_t+p_{\xi\xi}+cp_{\xi}+\sigma p+g(t,\xi,w^+)p] +\sigma \rho e^{-\sigma t} U         \nonumber\\
&& \le -\delta U (w^+-U)+\sigma \rho e^{-\sigma t}[-p_t+p_{\xi\xi}+cp_{\xi}+\sigma p+g(t,\xi,0)p] +\sigma \rho e^{-\sigma t} U\nonumber\\
&&\le \sigma \rho e^{-\sigma t}\left\{ -p_t+p_{\xi\xi}+cp_{\xi}+[\sigma+g(t,\xi,0)]p+\f{ -\delta U^2+\sigma U}{\sigma} \right\},\quad \xi:=x-ct.
\end{eqnarray*}
By the above constructions we see that $\mc{L}(w^+)\le 0$ for $\sigma\le \sigma_1$ and $\xi\not \in [\xi_0, \xi_0+1]$. Define
\[
\sigma_2:=\f{\delta (U(t,\xi_0+1))^2}{U(t,\xi_0)+\max_{\xi\in [\xi_0,\xi_0+1]}\{|p_t|+|p_{\xi\xi}|+c|p_\xi|+(\sigma_1+|g(t,\xi,0)|)p\}}.
\]
Then we have $\mc{L}(w^+)\le 0$ for all $\sigma\le \min\{\sigma_1,\sigma_2\}$ and $\xi\in [\xi_0,\xi_0+1]$. Thus, $w^+$ is a super solution for $\sigma\le \min\{\sigma_1,\sigma_2\}$ and $\rho\le \rho_0$.  Similarly, $w^-$ can be proved to be a sub-solution of \eqref{main-eq} for small $\rho$ and $\sigma$.
\end{proof}

Now we are in a position to prove Theorem  \ref{ThSS}.

\

\noindent
{\it Proof of Theorem \ref{ThSS}. }  By Lemma \ref{moderatespeed}, it follows that that \eqref{111}  holds. From Lemmas \ref{conv0} and \ref{largespeed}, we see that the fist two items and  \eqref{222} in the third item of Theorem \ref{ThSS} hold.  It then remains to prove that \eqref{333} holds. Indeed, we claim that 
\begin{equation}\label{claim-eq}
\lim_{t\to\infty}\sup_{x\in\R}|u(t,x)-U(t,x-ct)|=0.
\end{equation}
Let us postpone the proof of \eqref{claim-eq} and reach \eqref{333} in a few lines.  In view of \eqref{claim-eq} and Lemma \ref{exponential}, we apply the parabolic comparison principle to conclude that there exist $t_0>0$ and $C>0$ such that 
\begin{equation}\label{est111}
|u(t,x)-U(t,x-ct)|\le C e^{-\sigma t}, \quad \forall t\ge t_0.
\end{equation}

Now we return to the proof of the limit equality  \eqref{claim-eq}. Recall that $c\in (-c^*, c^*)$ and $\liminf_{x\to-\infty}u(0,x)>0$. Fix $\mu\in(-c, c^*)$. 
In view of  Lemma \ref{moderatespeed}, we only need to prove that 
\[
\lim_{t\to\infty}\sup_{x\le -\mu t}|u(t,x)-U(t,x-ct)|=0.
\]
Indeed, since $U(t,\xi)$ is periodic in $t\in\R$ and decreasing in $\xi\in\R$ with $U(t,-\infty)=\alpha(t)$, it follows that
\begin{eqnarray*}
\sup_{x\le -\mu t}|\alpha(t)-U(t,x-ct)|&&= \alpha(t)-\inf_{x\le -\mu t}U(t,x-ct)\nonumber\\
&&= \alpha(t)-\inf_{y\le 0} U(t, y-(\mu+c)t)\nonumber\\
&&\le \alpha(t)-U(t,-(\mu+c)t),
\end{eqnarray*}
which, together with the inequality $\mu+c>0$, the monotonicity of $U(t,\xi)$ in $\xi$ and the limit $U(t,-\infty)=\alpha(t)$, implies that 
\[
\lim_{t\to\infty}\sup_{x\le -\mu t}|\alpha(t)-U(t,x-ct)|=0.
\]
By the triangular inequality, it then suffices to prove 
\[
\lim_{t\to\infty}\sup_{x\le -\mu t}|u(t,x)-\alpha(t)|=0.
\]
Indeed, fix $\delta\in (0, \min\{c^*-c, \mu+c\})$. Let $x_0>0$ be specified later. Then we have 
\begin{eqnarray*}
\limsup_{t\to\infty}\sup_{x\le -\mu t}|u(t,x)-\alpha(t)| &&\le \limsup_{t\to\infty}\sup_{x\le -(\mu-\delta) t-x_0}|u(t,x)-\alpha(t)|\nonumber\\
&&=\limsup_{t\to\infty}\sup_{x\le -x_0}|u(t,x-(\mu-\delta) t)-\alpha(t)|.
\end{eqnarray*}
Set $w(t,x):= u(t,x-(\mu-\delta) t)$. Then $w$ solves
\begin{equation}\label{w-eq}
w_t=w_{xx}-(\mu-\delta)w_x+wg(t,x-(\mu-\delta+c)t, w)
\end{equation}
with $\liminf_{x\to-\infty}w(0,x)=\liminf_{x\to-\infty}u(0,x)>0$. Choose $M>\sup_{x\in\R} w(0,x)$. Then $\bar{w}(t)$, the solution of $\bar{w}'=\bar{w}g(t,-\infty, \bar{w})$ with $\bar{w}(0)=M$, is a 
super-solution of \eqref{w-eq} thanks to (G2). Consequently, the comparison principle implies that $w(t,x)\le \bar{w}(t)$ for all $t$ and $x$. Hence, 
\begin{equation*}
\limsup_{t\to\infty} \sup_{x\le -x_0}[w(t,x)-\alpha(t)]\le \limsup_{t\to\infty}[ \bar{w}(t)-\alpha(t)]=0.
\end{equation*}
It then remains to show that $\limsup_{t\to\infty} \sup_{x\le -x_0}[\alpha(t)-w(t,x)]\le 0$, that is, for any $\epsilon>0$ there exists $x_0>0$ and $t_0>0$ such that 
\begin{equation}\label{ineq-below}
 \inf_{x\le -x_0}w(t,x)>\alpha(t)-\epsilon,\quad t\ge t_0.
\end{equation}

We construct a family of sub-solutions to prove \eqref{ineq-below}. For any $\epsilon>0$, there exists $\gamma>0$ such that $\chi_\gamma(t)>\alpha(t)-\epsilon/2$, where $\chi_\gamma(t)$ is the unique positive periodic solution of $\chi'=\chi[g(t,-\infty,\chi)-\gamma]$. For the above $\gamma$ and $w(0,x)$, there exists $x_0>0$ such that
\[
g(t,-x_0,s)\ge g(t,-\infty, s)-\gamma, \quad t\in\R, s\in[0,\sup_{x\in\R} w(0,x)]
\] 
and 
\[
w(0,x)\ge \f{1}{2}\inf_{x\to-\infty}w(0,x),\quad x\le -x_0.
\]
For all $t\ge nT, x\le -x_0+(\mu-\delta+c)nT$ and $s\in[0,\sup_{x\in\R} w(0,x)]$, we have
\begin{equation*}
g(t,x-(\mu-\delta+c)t,s)\ge g(t,-x_0,s)\ge g(t,-\infty, s)-\gamma.
\end{equation*}

Let $x_n:=x_0 -(\mu-\delta+c)nT,\, \, \forall n \geq 0$, and  $\theta_n$ and $L$ be some positive numbers that will be specified later. Define
\begin{equation}\label{semi-line}
\phi^n(x):=\begin{cases}
\theta_n e^{(\mu-\delta) (x+x_n)}\sin\f{\pi}{L}(x+x_n+2L), & x\in [-x_n-L, -x_n],\nonumber\\
\theta_n e^{(\mu-\delta) (-L)}, & x< -x_n-L.
\end{cases}
\end{equation}
Consider a family of semi-line problems
 \begin{equation}
\begin{cases}
w_t^n=w_{xx}^n-(\mu-\delta)w_x^n+w^n[g(t,-\infty, w^n)-\gamma], &t>nT, x\le -x_n,\\
w^n(t,x)=0, &t>nT, x=-x_n.
\end{cases}
\end{equation}
By direct computations and the comparison principle, one may check that for all sufficiently small $\theta_n$, $w^n(nT+s,x; \phi^n)$ is convergent to the unique positive time-periodic solution $w^{n,*}(t,x)$ of \eqref{semi-line}. Moreover, $w^{n,*}(t,x)$ is non-increasing in $x$, $w^{n,*}(t, -\infty)=\chi_\gamma(t)$ and $w^{n,*}$ is increasing in $n$. Passing $n\to\infty$, we obtain that $w^*(t,x):=\lim_{n\to\infty} w^{n,*}(t,x)$ is a solution of the whole line problem 
\begin{equation*}
w_t^*=w_{xx}^*-(\mu-\delta)w_x^*+w^*[g(t,-\infty, w^*)-\gamma],  \,  \,  \forall x\in\R
\end{equation*} 
with
\begin{equation*}
w^*(t,-x_0):=\lim_{n\to\infty} w^{n,*}(t,-x_0)\ge w^{1,*}(t, -x_0)>0.
\end{equation*}
This, together with  Louville's theorem, leads to $w^*(t,x)\equiv \chi_\gamma(t)$, due to the fact that $-(\mu-\delta)\in (-c^*,c^*)$.
Meanwhile, for sufficiently small $\theta_n$ we have $w(nT,x)\ge \phi^n(x)$ for $x\le -x_n$. By using the comparison principle again, we then obtain
\begin{equation*}
\liminf_{n\to\infty}\sup_{x\le -x_0} w(nT+s,x)\ge \liminf_{n\to\infty}\sup_{x\le -x_0} w^n(nT+s, x)\ge  \lim_{n\to\infty} w^n(s, -x_0)=\chi_\gamma(s)
\end{equation*}
uniformly for $s\in[0,T]$. Hence, \eqref{ineq-below} is proved.

\section{An application}

As an application, we consider the following susceptible-infectious-susceptible epidemic model
\begin{equation}\label{SIS}
\begin{cases}
S_t=S_{xx}+B(t,N)N-\omega(t) SI-\mu(t,N)S+\gamma(t) I,\\
I_t=I_{xx}+\omega(t) SI-\mu(t,N)I-\gamma(t) I,
\end{cases}
\end{equation}
where $S$ represents the susceptible population density and $I$ the infectious population density, $N=S+I$ is the total population, function $B$ is the birth rate, $\omega$ is the transmission rate, $\mu$ is the death rate and $\gamma$ is the recovery rate. All time-dependent functions are assumed to be $T$-periodic and smooth for some $T>0$.

Adding $S$ and $I$ equations yields 
\begin{equation}\label{N}
N_t=N_{xx}+N[B(t,N)-\mu(t,N)].
\end{equation}
Once $N(t,x)$ is known, the infectious population satisfies 
\begin{equation}\label{I}
I_t=I_{xx}+I[\omega(t)N(t,x)-\mu(t,N(t,x))-\gamma(t)-\omega(t)I].
\end{equation}
In particular, $N(t,x)$ can be taken as a time periodic traveling wave solution of \eqref{N} with the form $N(t,x)=n(x-ct)$. As such, \eqref{I} is a special case of \eqref{main-eq} subject to suitable conditions. 

In the following, we apply Theorems \ref{ThTW} to study the existence of traveling wave solutions of \eqref{SIS}.  To ensure there is a KPP structure for the total population $N(t,x)$, we make the following assumption.
 \begin{enumerate}
 \item[(A1)] 	The functions $B$ and $\mu$ are smooth, and 
 $B(t,N)-\mu(t,N)$ is strictly decreasing in $N\in \R_+$;
 
 \item[(A2)]  $\int_0^T [B(t,0)-\mu(t,0)]dt>0$  and $B(t,M)-\mu(t,M) \leq 0$ for some $M>0$.
 \end{enumerate}

 By \cite[Theorem 5.2.1]{ZhaoBook},  we see   that the ordinary differential equation 
\begin{equation*}\label{N-ODE}
N'=N[B(t,N)-\mu(t,N)]
\end{equation*} 
admits a unique positive $T$-periodic state $N^*(t)$.  In view of   \cite[Lemma 4.1 and Theorem 4.2]{LYZ2006} with $\tau=0$ and $f(t,u,v)\equiv f(t,u)$,  
it easily follows that there is the minimal wave speed $c_N$ for periodic traveling wave solutions of \eqref{N} and 
\begin{equation*}\label{cn}
c_N=2\sqrt{\f{1}{T}\int_0^T [B(t,0)-\mu(t,0)]dt}.
\end{equation*}
Clearly, $N^*(t)$ and $c_N$ depends only on the birth rate $B$ and the death rate $\mu$. 
Substituting $N=N^*(t)$ into \eqref{I} and assuming that $I$ is independent of $x$, we obtain a time periodic ordinary differential equation
\begin{equation}\label{I-hom}
I'=I[A(t)-\omega(t)I],
\end{equation}
where 
\[
A(t):=\omega(t)N^*(t)-\mu(t,N^*(t))-\gamma(t).
\]
Let  $\bar{A}:=\f{1}{T}\int_0^T A(t)dt$. Then we have the following observation (see \cite[Theorem 5.2.1]{ZhaoBook}).
\begin{lemma}
Equation \eqref{I-hom} admits a unique positive periodic solution $I^*(t)$ if and only if $\bar{A}>0$. Moreover, $I^*(t)<N^*(t)$.
\end{lemma}

Applying Theorem \ref{ThTW} to equation \eqref{I} with $N(t,x)=n(t, x-ct), c\ge c_N$, we obtain the following result.

\begin{theorem}\label{Thapp}
Assume that (A1) and (A2) hold. Then system \eqref{SIS} admits a time periodic traveling wave connecting the endemic periodic state $(S^*(t), I^*(t))$ to $(0,0)$ with speed $c$ if and only if $\bar{A}>\f{(c_N)^2}{4}$ and $c\in[c_N, 2\sqrt{\bar{A}})$.
\end{theorem} 
In the rest of this section, we investigate how the parameters in system \eqref{SIS} influence $\bar{A}$. We focus on two factors: (i) amplitude of $\omega(t)$; (ii) oscillating period of time periodic parameters.

(i) We assume that $\omega(t)=l\tilde{\omega}(t)$ with $l>0$. Then $\bar{A}$ is strictly increasing in $l$ and there exist $0<l_*<l^*$ such that $\bar{A}>0$ if and only if $l>l_*$, and $\bar{A}>\f{(c_N)^2}{4}$ if and only if $l>l^*$. This, together with Theorem \ref{Thapp}, suggests that  for small transmission rate (i.e., $l<l_*$) the infectious population dies out  eventually and the susceptible population propagates with the speed $c_N$; for moderate transmission rate (i.e., $l\in (l_*, l^*)$) the susceptible population propagates with speed $c_N$, which is faster than the infectious; for large transmission rate (i.e., $l>l^*$) two populations propagate together with the same speed $c_N$.

(ii) We assume that $\omega(t)=\tilde{\omega}(\f{t}{T}), B(t,N)=\tilde{B}(\f{t}{T}, N), \mu(t,N)=\tilde{\mu}(\f{t}{T}, N), \gamma(t)=\tilde{\gamma}(\f{t}{T})$. Then $N^*(t)=N^*_T(t)$ satisfies
\[
\f{d}{dt}N_T^*(t)=N_T^*(t)\left[\tilde{B}\left(\f{t}{T}, N_T^*(t)\right)-\tilde{\mu}\left(\f{t}{T}, N_T^*(t)\right)\right].
\]
By the change of variables $t=Ts$ and $N_T^*(Ts)=v_T(s)$, it follows that
\begin{equation}\label{periodic-ode}
\begin{cases}
\f{d}{ds}v_T(s)=T v_T(s)[\tilde{B}(s,v_T(s))-\tilde{\mu}(s,v_T(s))],\\
v_T(s+1)=v_T(s)
\end{cases}
\end{equation}
and 
\begin{equation*}
\bar{A}=\int_0^1\left[ \tilde{\omega}(s)v_T(s) -\tilde{\mu}(s, v_T(s) )-\tilde{\gamma}(s) \right]ds.
\end{equation*}
Next we pass $T$ to $0$ and $\infty$, respectively,  to see how the oscillating period influences $\bar{A}$. For this purpose, we need the following result, which is of its own interest. 
\begin{lemma}\label{ode-limit}
Let $v_T, T>0$,  be the solution of \eqref{periodic-ode}. Then
$\lim_{T\to 0}v_T(s)=v_0$ and $\lim_{T\to\infty}[v_T(s)-v_\infty(s)]=0$ 
uniformly for $s\in[0,1]$, where $v_0$ is the unique zero of $\int_0^1 [\tilde{B}(s,w)-\tilde{\mu}(s,w)]ds$ and $v_\infty(s)$ is the unique zero of $\tilde{B}(s,w)-\tilde{\mu}(s,w)$. 
\end{lemma}
\begin{proof}
  We first consider the case where $T\to 0$. By assumptions (A1) and (A2), we see that there exist $M>\delta>0$ such that $\bar{v}\equiv M$ is a sup-solution and $\underline{v}\equiv \delta$ is a sub-solution of \eqref{periodic-ode} for all $T>0$. By the uniqueness of positive periodic solutions of \eqref{periodic-ode}
  (see \cite[Theorem 5.2.1]{ZhaoBook}), we obtain $\delta\le v_T(s)\le M$ for all $T>0$ and $s\in[0,1]$. Hence, both $v_T'(s)$ and $v_T''(s)$ are uniformly bounded in $T$ and $s$. Therefore, passing $T\to 0$ in \eqref{periodic-ode} yields, up to subsequence, 
\[
\lim_{T\to 0}v_T(s)\equiv C, \,  \,  \text{uniformly for } s\in [0,1] ,  
\text{    for some    } C\in [\delta, M].
\]
Meanwhile,
\[
0\equiv \f{1}{T} \int_0^1 \f{v_T'(s)}{v_T(s)}ds=\int_0^1[ \tilde{B}(s,v_T(s))-\tilde{\mu}(s,v_T(s))]ds,
\]
in which, passing $T\to 0$, we see that $\int_0^1 [\tilde{B}(s,C)-\tilde{\mu}(s,C)]ds=0$. This implies that  $C=v_0$.

Next we consider the case where $T\to \infty$. Given $\epsilon>0$, we define $\bar{v}(s)=v_\infty(s)+\epsilon$. We claim $\bar{v}_\epsilon(s)$ is a sup-solution of \eqref{periodic-ode} when $T$ is sufficiently large. Indeed, By assumptions (A1) and (A2), we see that for such $\epsilon$ there exists $\delta>0$ such that $\tilde{B}(s,\bar{v}_\epsilon)-\tilde{\mu}(s,\bar{v}_\epsilon)>\delta$ for all $s\in[0,1]$.
\[
\bar{v}_\epsilon'-T\bar{v}_\epsilon[\tilde{B}(s,\bar{v}_\epsilon)-\tilde{\mu}(s,\bar{v}_\epsilon)]=v_\infty'-T\bar{v}_\epsilon[\tilde{B}(s,\bar{v}_\epsilon)-\tilde{\mu}(s,\bar{v}_\epsilon)]<v_\infty'-TM\delta,
\]
which is negative under the condition  that 
\[
T>\f{\max_{s\in[0,1]}|v_\infty'(s)|}{M\delta}.
\]
Thus, $\bar{v}_\epsilon(s)$ is a sup-solution of \eqref{periodic-ode} for all
sufficiently large $T$. Similarly, $\underline{v}(s):=v_\infty-\epsilon$ is a 
sub-solution for all sufficiently large $T$. It then follows that the unique positive solution $v_T(s)$ satisfies
\[
\underline{v}_\epsilon(s)\le v_T(s)\le \bar{v}_\epsilon(s), \quad \forall T\ge T_0(\epsilon),
\]
where $T_0(\epsilon)$ is a large number depending on $\epsilon$. This shows that  $\lim_{T\to\infty}[v_T(s)-v_\infty(s)]=0 $ uniformly in $s\in[0,1]$.
\end{proof}

With Lemma \ref{ode-limit} above, we immediately obtain
\begin{equation*}\label{Tto0}
\lim_{T\to 0}\bar{A}= \int_0^1 [\tilde{\omega}(s)v_0-\tilde{\mu}(s, v_0 )-\tilde{\gamma}(s) ]ds,
\end{equation*}
which implies that as the oscillating period shrinks to zero, the speed interval $[c_N, 2\sqrt{\bar{A}})$ is the same as that of the average system.  While as the oscillating period increases to infinity, we have
\begin{equation}\label{Ttoinfty}
\lim_{T\to \infty }\bar{A}=\int_0^1 [\tilde{\omega}(s)v_\infty(s)-\tilde{\mu}(s, v_\infty(s) )-\tilde{\gamma}(s) ]ds,
\end{equation}
where $\tilde{\omega}(s)$ and $v_\infty(s)$ are independent from each other.  Finally, we use an example to illustrate that $\lim_{T\to \infty }\bar{A}$ can be either greater or less than that of the average system. Indeed, choose $\tilde{\mu}(s,v)=\delta v$. Then in the right-hand side of  \eqref{Ttoinfty},  the integral of the last two terms are the average of $-\delta v_\infty(s)-\tilde{\gamma}(s)$. For the first term $\int_0^1 [\tilde{\omega}(s)v_\infty(s)ds$, we can infer that
\begin{equation*}
\sup_{\tilde{\omega}\in \Lambda}\int_0^1 \tilde{\omega}(s)v_\infty(s)ds=\max_{s\in[0,1]}v_\infty(s) ,\quad \inf_{\tilde{\omega}\in \Lambda}\int_0^1 \tilde{\omega}(s)v_\infty(s)ds=\min_{s\in[0,1]}v_\infty(s), 
\end{equation*}
where 
\begin{equation*}
\Lambda:=\left\{\tilde{\omega}\in C(\R,\R_+):\, \,  \tilde{\omega}(s)=\tilde{\omega}(s+1), \int_0^1\tilde{\omega}(s)ds=1 \right\}.
\end{equation*}



\begin{thebibliography}{99}
\bibitem{ABR-preprint}M. Alfaro, H. Berestycki, G. Raoul, The effect of climate shift on a species submitted to dispersion, evolution, growth and nonlocal competition, {\it SIAM J. Math. Anal.}, 49 (2017), 562-596.

\bibitem{AW-AM78}D.G. Aronson and H.F. Weinberger, Multidimensional nonlinear diffusions arising in population genetics, {\it Adv. Math.}, 30 (1978), 33-76.

\bibitem{BDNZ-BMB09} H. Berestycki, O. Diekmann, C. Nagelkerke and P.  Zegeling, Can a species keep pace with a shifting climate? 
{\it Bull. Math. Biol.}, 71 (2009), 399-429.

\bibitem{BF18}H. Berestycki and J. Fang, Forced waves of the Fisher-KPP equation in a shifting environment, {\it J. Diff. Eqns.}, 264 (2018), 2157-2183,

\bibitem{BHR-AMRA07} H. Berestycki, F. Hamel and L. Rossi, Liouville-type results for semilinear elliptic equations in unbounded domains. {\it Ann. Mat. Pura Appl.}, 186 (2007), 469-507.



%

\bibitem{BR-JEMS06}H. Berestycki and L. Rossi, On the principal eigenvalue of elliptic operators in $R^N$ and applications. {\it J. Eur. Math. Soc. (JEMS)}, 8 (2006), 195-215. 

\bibitem{BR-DCDS08}H. Berestycki and L. Rossi, Reaction-diffusion equations for population dynamics with forced speed. I. The case of the whole space. {\it Discrete Contin. Dyn. Syst.},  21 (2008),  41-67.

\bibitem{BR-DCDS09}H. Berestycki and L. Rossi, Reaction-diffusion equations for population dynamics with forced speed. II. Cylindrical-type domains. {\it Discrete Contin. Dyn. Syst.},  25 (2009), 19Ð61.

\bibitem{BR-CPAM15}H. Berestycki and L. Rossi, Generalizations and properties of the principal eigenvalue of elliptic operators in unbounded domains. {\it Comm. Pure Appl. Math.}, 68 (2015), 1014-1065.


\bibitem{BG-2016preprint}J. Bouhours and T. Giletti, Spreading and vanishing for a monostable reaction-diffusion equation with forced speed, {\it J. Dyn. Diff. Eqns.}, 31(2019), 247-284.


\bibitem{BL-preprint}J. Bouhours and M. Lewis, Climate change and integrodifference equations in a stochastic environment, {\it Bull. Math. Biol.,} 78 (2016),  1866-1903.



\bibitem{Bramson-83}M.D. Bramson, Convergence of solutions of the Kolmogorov equation to travelling waves, {\it Mem. Amer. Math. Soc.}, 44 (1983).

\bibitem{CTW17} X. Chen, J.-C. Tsai and Y. Wu, Longtime behavior of solutions of an SIS epidemiological model, {\it SIAM J. Math. Anal.}, 49 (2017), 3925-3950,



\bibitem{DWZ-preprint}Y. Du, L. Wei and L. Zhou, Spreading in a shifting environment modeled by the diffusive logistic equation with a free boundary, {\it J. Dyn. Diff. Eqn.,} 30(2018), 1389-1426.

\bibitem{FLW-SIAP16}J. Fang, Y. Lou and J. Wu, Can pathogen spread keep pace with its host invasion? {\it SIAM J. Appl. Math.}, 76 (2016), 1633-1657.

\bibitem{Fisher}R.A. Fisher, The wave of advance of advantageous genes, {\it Ann. Eugenics}, 7 (1937), 353Ð369.

\bibitem{HR-JEMS11}F. Hamel and L. Roques, Uniqueness and stability properties of monostable pulsating fronts. {\it J. Eur. Math. Soc. (JEMS)}, 13 (2011), 345-390.


\bibitem{H-AIHP97-2} F. Hamel, Reaction-diffusion problems in cylinders with no invariance by translation. II. Monotone perturbations,  {\it Ann. Inst. H. Poincar\'e Anal. Non Lin\'eaire}, 14 (1997), 555-596. 

\bibitem{ZhouKot-14}M. Harsch, Y. Zhou, J. Hille Ris Lambers and M. Kot, Keeping pace with climate change: stage-structured moving-habitat models, {\it The American Naturalist},  184(2014), 25-37.

\bibitem{HS-SIAM14}M. Holzer and A. Scheel, Accelerated Fronts in a Two-Stage Invasion Process, {\it SIAM J. Math. Anal.}, 46(2014), 397-427. 

\bibitem{HL-JDE15}C. Hu and B. Li, Spatial dynamics for lattice differential equations with a shifting habitat, {\it J. Diff. Eqns.}, 259(2015), 1967-1989.

\bibitem{HZ17}H. Hu and X. Zou, Existence of an extinction wave in the Fisher equation with a shifting habitat, {\it Proc. Amer. Math. Soc.}, 145 (2017), 4763-4771.

\bibitem{LeiDu-DCDS17}C. Lei and Y. Du, Asymptotic profile of the solution to a free boundary problem arising in a shifting climate model, {\it Discrete Cont. Dyn. Syst. B}, 22(2017), 895-911.

\bibitem{LYZ2006} X. Liang, Y. Yi and X.-Q. Zhao, Spreading speeds and traveling waves for periodic evolution systems,
{\it J. Differential Equations}, 231(2006), 57-77.

\bibitem{KPP}A.N. Kolmogorov, I.G. Petrovsky and N.S. Piskunov, \'Etude de l'\'equation de la diffusion avec croissance de la quantit\'e de mati$\grave{\text{e}}$re et son application $\grave{\text{a}}$ un probl$\grave{\text{e}}$me biologique, {\it Bull. Univ. \'Etat Moscou, S\'er. Inter., } A 1 (1937), 1-26.

\bibitem{LMS}M. A. Lewis, N. G. Marculis, and Z. Shen, Integrodifference equations in the presence of climate change: persistence criterion, travelling waves and inside dynamics, {\it J. Math. Biol.}, 77(2018), 1649-1687.

\bibitem{LBSF-SIAP14}B. Li, S. Bewick, J. Shang and W. Fagan, Persistence and spread of a species with a shifting habitat edge. {\it SIAM J. Appl. Math.,} 74 (2014), 1397-1417.

\bibitem{LBBF16}B. Li, S. Bewick, M. R. Barnard, and W. F. Fagan, Persistence and spreading speeds of integro-difference equations with an expanding or contracting habitat, {\it Bull. Math. Biol.}, 78 (2016), 1337-1379.

\bibitem{LWZ18}W.-T. Li, J.-B. Wang and X.-Q. Zhao, Spatial dynamics of a nonlocal dispersal population model in a shifting environment, {\it J. Nonlinear Sci.}, 28 (2018), 1189-1219.

\bibitem{PM-BMB04}A. B. Potapov and M. A. Lewis. Climate and competition: the effect of moving range boundaries on habitat invasibility, {\it Bull. Math. Biol.}, 66(2004), 975-1008.


\bibitem{Vo-JDE15}H.-H. Vo, Persistence versus extinction under a climate change in mixed environments, {\it J. Diff. Eqns.}, 259(2015), 4947-4988.

\bibitem{WZ18}J.-B. Wang and X.-Q. Zhao, Uniqueness and global stability of forced waves in a shifting environment, {\it Proc. Amer. Math. Soc.}, 147 (2019), 1467-1481.

\bibitem{WeiZhangZhou-CVPDE16}L. Wei, G. Zhang and M. Zhou, Long time behavior for solutions of the diffusive logistic equation with advection and free boundary, {\it Cal. Var. \& PDEs.,} 55(2016), No. 95, 34 pp. 


\bibitem{ZhaoBook}X.-Q. Zhao, {\it Dynamical Systems in Population Biology}, second edition, Springer-Verlag, New York, 2017.

\bibitem{ZK-Springer13}Y. Zhou and M. Kot, Life on the move: modeling the effects of climate-driven range shifts with integrodifference equations, In {\it Dispersal, Individual Movement and Spatial Ecology, Vol.  2071, } Lecture Notes in Math.,  263-292. Springer, Heidelberg, 2013.


\end{thebibliography}
\end{document}